\documentclass[12pt, a4paper, headsepline]{scrartcl} %or this?

\usepackage{ifthen}       % used for ``draft'' mode

\ifthenelse{\isundefined \AMSstylefile } % if not AMSstylefile
{
  \usepackage{scrpage2}    % change headings, see Franzis book, p. 500)
  \pagestyle{scrheadings}
  \typearea[7.5mm]{17}    % [BCOR Bindungskorrektur]{DIV Einteilung der Seite}
}
{}

\usepackage{scrtime} % give time by \thistime
\usepackage[dvipsnames]{xcolor}

\newcommand{\markerO}{\fbox{\rule{0pt}{0.1ex}\textbf{Olaf}}}
\newcommand{\look}[1]{\markerO \textbf{*}%\rule{3ex}{2.5mm}
    \footnote{ #1 }}
\newcommand{\lookO}[1]{\markerO\textbf{*}%\rule{3ex}{2.5mm}
    \footnote{\textbf{Olaf:} #1 }}

\ifthenelse{\isundefined \draft }
{
  \renewcommand{\look}[1]{}%   final version
  \renewcommand{\lookO}[1]{}%

  \ifthenelse{\isundefined \AMSstylefile } % if not AMSstylefile
  {\automark[subsection]{section}}  %
  {}
}
{
  \ifthenelse{\isundefined \Olaf}  % draft version
  {\usepackage[notref,notcite]{showkeys}}% on other computers
  {\usepackage{/home/post/Aktuell/LaTeX/my-showkeys}}% on Olaf's computer
  
  \ifthenelse{\isundefined \AMSstylefile } % if not AMSstylefile
  {
    \rehead{\parbox[t]{0.3\textwidth}%
      {\texttt{\jobname.tex}\\\today, \thistime}}
    \lohead{\parbox[t]{0.3\textwidth}%
      {\texttt{\jobname.tex}\\\today, \thistime}}
  }
  {}
}
\usepackage{amsmath, amssymb}%, theorem}
\usepackage{amsthm}

\usepackage{amscd}

\usepackage{accents}     % for \accentset, \ring...
\usepackage{yhmath}     % for \wideparen, ... (wide parenthesis over symbols)

\usepackage{bbm}         % blackboard (see: /usr/share/texmf/mytex/bbm.dvi)

\usepackage{mathrsfs}    % use ``real'' caligraphic letters for \mathcal
\renewcommand\mathcal\mathscr  

\usepackage[german,english]{babel}   % last language determines
\usepackage[T1]{fontenc}
\usepackage[utf8]{inputenc}

\numberwithin{equation}{section}

\newcounter{myenumi}

\newenvironment{myitemize}
{
  \begin{list}{$\bullet$}
  {
    \setlength{\topsep}{\smallskipamount}
    \setlength{\itemsep}{0.0em}
    \setlength{\leftmargin}{\parindent}
    \setlength{\labelwidth}{1em}
    \setlength{\labelsep}{0.5em}
  }
}
{
  \end{list}
}

\newtheoremstyle{mythmstyle}% name
  {\topsep}% Space above
  {\topsep}% Space below
  {\itshape}% Body font
  {}% Indent amount
  {\bfseries \sffamily}% Theorem head font
  {.}%Punctuation after theorem head
  {.5em}%Space after theorem head
  {}% theorem head spec

\newtheoremstyle{mydefstyle}% name
  {\topsep}% Space above
  {\topsep}% Space below
  {\normalfont}% Body font
  {}% Indent amount
  {\bfseries \sffamily}% Theorem head font
  {.}%Punctuation after theorem head
  {.5em}%Space after theorem head
  {}% theorem head spec

  \ifthenelse{\isundefined \Olaf} % zwei Cappuccini fuer Olaf ;-)
{} % general version
{} % Olaf's version
\swapnumbers                    % 1.1 Theorem instead of Theorem 1.1
\theoremstyle{mythmstyle}       % my style (new fonts) -- body italics

\newtheorem{theorem}{Theorem}[section]

\newtheorem{proposition}[theorem]{Proposition}
\newtheorem{lemma}[theorem]{Lemma}
\newtheorem{corollary}[theorem]{Corollary}

\theoremstyle{mydefstyle}        % my style (new fonts) -- body roman

\newtheorem{assumption}[theorem]{Assumption}
\newtheorem{example}[theorem]{Example}

\newtheorem{remark}[theorem]{Remark}

\newtheorem*{remark*}{Remark}

\expandafter\let\expandafter\oldproof\csname\string\proof\endcsname
\let\oldendproof\endproof
\renewenvironment{proof}[1][\bfseries\sffamily\proofname]{%
  \oldproof[\bfseries \sffamily #1]%
}{\oldendproof}

\newcommand{\abssqr}[2][{}]{\lvert{#2}\rvert^2_{#1}} % abs squared
\newcommand{\normsymb}{|\!|}

\newcommand{\norm}[2][{}]{\normsymb{#2}\normsymb_{{#1}}}    % norm
\newcommand{\normsqr}[2][{}]{\normsymb{#2}\normsymb^2_{#1}} % norm squared

\newcommand{\iprod}[3][{}]{\langle{#2},{#3}\rangle_{#1}}  % inner product

\newcommand{\set}[2]{\{ \, #1 \, | \, #2 \, \} }      % set { #1 | #2 }

\newcommand{\map}[3]{ #1 \colon #2 \longrightarrow #3}    % maps
\newcommand{\bd}  {\partial}          % symbol for boundary of a set
\newcommand{\restr}[1]{{\restriction}_{#1}} % symbol for map restriction

\newcommand{\dd}    {\, \mathrm d}    % not optimal: no \, if at beginning
\DeclareMathOperator{\dom}    {dom}
\DeclareMathOperator{\ran}    {ran}
\DeclareMathOperator{\id}     {id}   % identity map
\DeclareMathOperator{\supp}   {supp}

\DeclareMathOperator{\tr}     {tr}  % trace

\DeclareMathOperator{\Span}    {span}

\newcommand{\specsymb} {\sigma} % symbol for spectrum

\newcommand{\spec}[2][{}]   {\specsymb_{\mathrm{#1}}(#2)}

\renewcommand{\phi}{\varphi}   % shortcut
\renewcommand{\rho}{\varrho}   % shortcut

\renewcommand{\Re}     {\mathrm {Re}\,}
\newcommand{\conj}[2][{}]{\overline {{#2}}^{#1}} 
\newcommand{\R}{\mathbb{R}} % symbol for real numbers
\newcommand{\C}{\mathbb{C}} % symbol for complex numbers
\newcommand{\1}{\mathbbm 1}                    % blackboard 1

\newcommand{\im}{\mathrm i} % complex unit

\newcommand{\wt}{\widetilde}           % shortcut
\newcommand {\qf}[1]{\mathfrak{#1}}    % font for quadratic forms
\newcommand{\HS}{\mathcal H}           % symbol for Hilbert space
\newcommand{\HSaux}{\mathcal G}        % symbol for auxiliary Hilbert space

\newcommand{\Sobsymb} {\mathsf H}      % symbol for Sobolev space
\newcommand{\Lsymb}    {\mathsf L}     % symbol for int L-spaces
\newcommand{\Sobspace}[1][1]{\Sobsymb^{#1}} 
\newcommand{\Lpspace}[1][p]    {\Lsymb_{#1}}     % symbol for int L-spaces
\newcommand{\Lsqrspace}    {\Lpspace[2]}     % symbol for int L-spaces
\newcommand{\Lp}[2][p]{\Lpspace [#1]({#2})} % L_#1(#2)-spaces

\newcommand{\Lsqr}[2][{}]{\Lsqrspace^{#1}({#2})} % L_2^{#1}(#2)-spaces
\newcommand{\Sob}[2][1]{\Sobspace [#1]({#2})}         % Sobolev space
\newcommand{\Neu}{{\mathrm N}}              % symbol for Neumann bd cond
\newcommand{\Dir}{{\mathrm D}}              % symbol for Dirichlet bd cond
\newcommand{\quadtext}[1]{\quad\text{#1}\quad}

\newcommand{\HDir}{H^\Dir}   % Dirchlet operator (with label)
\newcommand{\wtHDir}{\widetilde H^\Dir}   % Dirchlet operator (with label)
\newcommand{\RDir}{R^\Dir}   % Resolvent of Dirichlt operator (with label)

\newcommand{\HNeu}{H}   % Neumann operator (without label)
\newcommand{\RNeu}{R}   % Neumann operator (without label)

\newcommand{\LS}{\mathcal N}           % symbol for ``loesungs''space
\newcommand{\dplus}{\mathop{\dot+}}

\begin{document}

\title%[Generalised half-space problems]
{On a generalisation of Krein's example}

\ifthenelse{\isundefined \AMSstylefile } % if not AMSstylefile
{
\author{Olaf Post	\qquad	Christoph Uebersohn}

\date{\today, \thistime,  \emph{File:} \texttt{\jobname.tex}}
}
{
\author{Olaf Post (\today)}
\address{Fachbereich 4 -- Mathematik,
         Universit\"at Trier,
         54286 Trier, Germany}
         \email{olaf.post@uni-trier.de}
         
\author{Christoph Uebersohn (\today)}
\address{Fachbereich 8 -- Institut f\"ur Mathematik, 
				 Johannes Gutenberg-Universit\"at Mainz,
         D-55099 Mainz, Germany}
         \email{uebersoc@uni-mainz.de}
}

\ifthenelse{\isundefined \draft}
{\date{\today}}  % final version
{\date{\today, \thistime,  \emph{File:} \texttt{\jobname.tex}}} 
         % draft version

\maketitle

%-------------------------------------------------------------
% Abstract.
%-------------------------------------------------------------
\begin{abstract}
  We generalise a classical example given by Krein in 1953. 
	We compute the difference of the resolvents and 
	the difference of the spectral projections explicitly. 
	We further give a full description of the unitary invariants,
	i.\,e., of the spectrum and the multiplicity.
  Moreover, we observe a link between the difference of 
	the spectral projections and Hankel operators.
\end{abstract}

%------------------------------------------------------------
% Subject classifications
%------------------------------------------------------------

%\subjclass[2010]{Primary 47B15; Secondary 47A55, 35J25, 47A10, 47B35}
2010~\emph{Mathematics Subject Classification}.
Primary 47B15; Secondary 47A55, 35J25, 47A10, 47B35.

\emph{Key~words~and~phrases}. Self-adjoint operators, spectral projections, 
	perturbation theory, boundary pairs, spectrum and resolvent, 
	Hankel operators.
%\keywords{Self-adjoint operators, spectral projections, 
%	perturbation theory, boundary pairs, spectrum and resolvent, 
%	Hankel operators}

%-----------------------------------------------------------------------
%
% 111
\section{Introduction and main results}
\label{sec:intro}
%
%-----------------------------------------------------------------------

\subsection{Introduction}

Krein presented in~\cite{Krein} a rigorous definition of the spectral
shift function $\xi=\xi({\bullet},A_1,A_0) \in \Lp[1] \R$ defined via
\begin{equation*}
  %\label{eq:ssf1}
  \tr(\chi(A_1)-\chi(A_0)) 
  = \int_\R \chi^{\prime}(\vartheta) \xi(\vartheta) \dd \vartheta,
\end{equation*}
whenever $\chi$ belongs to a suitable class of functions and $A_1-A_0$
is of trace class.  In a naive definition, one would choose the
indicator function $\chi=\1_{(-\infty,\vartheta)}$, as the above formula
then becomes formally
\begin{equation}
  \label{eq:ssf2}
  \tr\bigl(\1_{(-\infty,\vartheta)}(A_1)-\1_{(-\infty,\vartheta)}(A_0)\bigr) 
  = \xi(\vartheta).
\end{equation}
Unfortunately,\footnote{Maybe this is a fortune as it gave rise to new
  research \dots} formula~\eqref{eq:ssf2} is not true: even if
$A_1-A_0$ is a rank $1$ perturbation (and hence of trace class), the
difference of the spectral projections $\1_{(-\infty,\vartheta)}(A_1)
- \1_{(-\infty,\vartheta)}(A_0)$ need not to be of trace class, i.\,e.,
the left hand side of~\eqref{eq:ssf2} is not defined.  Krein presented
such an example in his paper~\cite{Krein}, where $A_1=(H+1)^{-1}$ and
$A_0=(H^\Dir+1)^{-1}$ are the resolvents at the spectral point $-1$
of the Neumann and Dirichlet Laplacian $H=\bigr({-}\frac {\dd^2}{\dd
  t^2}\bigr)^\Neu$ and $H^\Dir=\bigr({-}\frac {\dd^2}{\dd
  t^2}\bigr)^\Dir$ on the half-line $\R_+ = (0,\infty)$, respectively.
Krein showed that the difference of the spectral projections is an
integral operator given by
\begin{equation}
  \label{eq:spec.proj.class1}
  \bigg( \Big[ 
  \1_{(-\infty,\vartheta)}(A_0)-\1_{(-\infty,\vartheta)}(A_1)
  \Big] 
  \psi \bigg)(t)	
  = \frac 2 \pi
  \int_{\R_+} 
  \frac{\sin((\frac{1}{\vartheta} - 1)^{1/2} (t+\tau))}{t+\tau}
  \psi(\tau) \dd \tau
\end{equation}
for $t \in \R_+$, $0<\vartheta<1$ and $\psi \in \mathsf C_{\mathrm
  c}(\R_+)$, and hence not Hilbert Schmidt.  Kostrykin and Makarov
diagonalised the integral operator of~\eqref{eq:spec.proj.class1} and
proved that it has a simple purely absolutely continuous spectrum
filling in the interval $[{-}1,1]$; in particular, the integral
operator of~\eqref{eq:spec.proj.class1} is not compact,
see~\cite{Kostrykin_Makarov}.  Note that the kernel function of the
integral operator of~\eqref{eq:spec.proj.class1} depends only on the
sum of the variables; such operators on $\Lp[2] {\R_{+}}$ are called
\emph{Hankel (integral) operators}.  We refer to Peller's
monograph~\cite{Peller_II} for an overview on Hankel operators.

Relations between differences of spectral projections and
Hankel operators are also discussed in the work of
Pushnitski~\cite{Pushnitski_I,Pushnitski_II, Pushnitski_III} and
together with Yafaev~\cite{Pushnitski_Yafaev,Pushnitski_Yafaev_II} in
the framework of scattering theory, related to an idea of
Peller~\cite{Peller_I}.  We also refer to~\cite{Uebersohn} for an
approach based on a result of Megretski{\u\i}, Peller, and
Treil~\cite{Megretskii_et_al}.

In this paper we generalise Krein's example by considering
operators of the type 
\begin{align}
  \label{Gen.Neu/Dir}
	H = \Bigl({-}\frac{\dd^2}{\dd t^2}\Bigr)^{\Neu}
	\otimes \id + \id \otimes L
	\quadtext{and}
	H^{\Dir} = \Bigl({-}\frac{\dd^2}{\dd t^2}\Bigr)^{\Dir}
	\otimes \id + \id \otimes L
	\quadtext{in}
	\Lp[2]{\R_{+}} \otimes \mathcal{G},
\end{align}
where $ \mathcal{G} \neq \{ 0 \} $ is a separable complex Hilbert
space and $ L $ is a self-adjoint nonnegative operator on $
\mathcal{G} $ (precise definitions are given in
Section~\ref{sec:bd2.half-space}).  We call $H$ resp.\ $\HDir$ the
\emph{(abstract) Neumann} resp.\ \emph{Dirichlet operator}.  In
particular, this framework includes
\begin{enumerate}
\item Krein's example of the half-line $ \R_{+}$ with $L=0$ and
  $\mathcal G=\C$;
\item the example of the classical half-space $ \R_{+} \times \R^{n-1}
  $ with $L=-\Delta_{\R^{n-1}}$ and $ n \geq 2 $;
\item the case when $L$ is (minus) the Laplacian on a generally
  noncompact manifold $Y$, e.\,g. on the cylinder $\R_{+} \times Y$
  with Neumann resp.\ Dirichlet boundary conditions on $\{0\}\times
  Y$.
\end{enumerate}
We consider the resolvents
\begin{align}	\label{Res.D/N_at_-1}
	A_0 = (\HDir + 1)^{-1} 
	\quad	\text{and}	\quad 
	A_1 = (H + 1)^{-1} 
\end{align}
of the operators $ H^{\Dir} $ and $ H $ defined in~\eqref{Gen.Neu/Dir}
at the spectral point $-1$.  The difference $A_1-A_0$ of the
resolvents will be computed with the help of a Krein-type resolvent
formula from the theory of boundary pairs~\cite{Post_I}.

Next we would like to compute the difference 
$ \1_{(-\infty, \vartheta)}(A_0)-\1_{(-\infty, \vartheta)}(A_1) $ 
of the spectral projections for all $0<\vartheta<1$.
It is generally hard to compute 
differences of spectral projections explicitly. 
In our example, however, the computation can be performed,
using the transformation formula for spectral measures
and the above mentioned convolution-type formula from~\cite{Weidmann}.
This idea is borrowed from Krein's example.

We give a full description of the unitary invariants of the resolvent
difference and of the difference of the spectral projections.
Moreover, the spectral properties establish a link between the
difference of the spectral projections and Hankel operators.

Operators of the type~\eqref{Gen.Neu/Dir} have been studied before;
criteria for self-adjointness (see, e.\,g., Schmüdgen's
monograph~\cite{Schmuedgen}), the spectrum (see,
e.\,g.,~\cite{Schmuedgen} or Weidmann's monograph~\cite{Weidmann}),
and a convolution-type formula for the spectral projection
(see~\cite{Weidmann}) are known and will be very useful in this paper.
There are classical works on spectral theory of self-adjoint boundary
value problems with operator-valued potential as
in~\eqref{Gen.Neu/Dir}, see, e.\,g., Gorbachuk and
Kutovoi~\cite{Gorbachuk,Gorbachuk_II, Gorbachuk_Kutovoi,
  Gorbachuk_Kutovoi_II,Kutovoi} and the
monograph~\cite{gorbachuk-gorbachuk:91}.  Gorbachuk and Kutovoi showed
in~\cite{Gorbachuk_Kutovoi} that $ A_1-A_0 $ is trace class if and
only if (in the present notation) $ (L+1)^{-1} $ is trace
class. Sufficient criteria for $ A_1-A_0 $ to belong to Schatten
classes can be found in~\cite{Gorbachuk_Kutovoi_II}.  The proofs rely
on the resolvent identities and the ideal properties of Schatten
classes; the resolvent difference is not computed explicitly
in~\cite{Gorbachuk_Kutovoi,Gorbachuk_Kutovoi_II}.

Abstract boundary value problems have often been treated using
operator theory.  We refer to the review article~\cite{dhms:12} for an
overview on boundary triplets and also to~\cite{Post_I} for the
concept of boundary pairs, see also the references therein.  Such
concepts allow for example to calculate differences of resolvents of
operators with different boundary conditions.  There are related works
by Boitsev, Neidhardt, and Popov~\cite{Boitsev_Neidhardt_Popov} on
tensor products of boundary triplets (with bounded operator $L$),
Malamud and Neidhardt~\cite{Malamud_Neidhardt} for unitary equivalence
and regularity properties of different self-adjoint realisations,
Gesztesy, Weikard, and Zinchenko~\cite{Gesztesy_Weikard_Zinchenko__I,
  Gesztesy_Weikard_Zinchenko__II} for a general spectral theory of
Schr\"odinger operators with bounded operator potentials, and
Mogilevskii~\cite{Mogilevskii}, see also the references therein.
Moreover, when finishing this paper, the authors of the present paper
have learned about the recent paper~\cite{Boitsev_et_al}, where
Boitsev, Brasche, Malamud, Neidhardt and Popov construct a boundary
triplet for the adjoint of the symmetric operator $ T \otimes \id +
\id \otimes L $ with $ T $ being symmetric and $ L $ being
self-adjoint.  This generalises the situation of~\eqref{Gen.Neu/Dir},
where $ T = - \dd^{2}/\dd t^{2} $ on $ \Lsqr{\R_{+}} $.  The focus
in~\cite{Boitsev_et_al} is on self-adjoint extensions which do not
respect the tensor structure~\eqref{Gen.Neu/Dir} as models for quantum
systems coupled to a reservoir.  Note that
in~\cite{Malamud_Neidhardt,Boitsev_et_al} one has to ``regularise''
the boundary triplet (i.e., one has to modify the boundary map and
spectrally decompose $L$ into bounded operators) in order to treat
also unbounded operators $L$.  In our approach, we can directly treat
unbounded operators $L$ without changing the boundary map or
decomposing $L$.  The special case of operators $L$ with purely
discrete spectrum has been treated e.\,g.\,in~\cite[Sec.~6.4]{Post_I}
or in a slightly different setting in~\cite[Sec.~3.5.1]{post:12}.

The results of this paper will be part of the PhD thesis of the second
author at Johannes Gutenberg University Mainz.

%-----------------------------------------------------------------------
\subsection{Main results}
%\label{subsec:main.results}
%-----------------------------------------------------------------------

Let $A_0$ and $A_1$ be the resolvents defined in~\eqref{Res.D/N_at_-1}
of the (abstract) Dirichlet and Neumann operators given
in~\eqref{Gen.Neu/Dir} above.
\begin{theorem}
  \label{thm:main1}
	\begin{enumerate}
		\item \label{thm:main1__Teil(a)}	
			The resolvent difference $A_1-A_0$ 
			acts on elementary tensors as follows:
			\begin{align*}
				\big( [A_1-A_0] (\psi \otimes \varphi) \big)(t)	
				= \int_{\mathbb{R}_{+}} 
				\psi(\tau) \exp \! \big( {-} (L+1)^{1/2} (t+\tau) \big)
				(L+1)^{-1/2} \varphi \dd \tau
			\end{align*}
			for all $ \psi \in \mathsf C_{\mathrm c}(\mathbb{R}_{+}) $ 
			and all $ \varphi \in \mathcal{G} $.
		\item \label{thm:main1__Teil(b)}
			Let $0<\vartheta<1$ and 
			let $\alpha(\vartheta) = \frac{1}{\vartheta} - 1 >0$. 
			Then the difference of 
			the spectral projections of $ A_0 $ and $ A_1 $ associated with 
			the open interval $ (-\infty,\vartheta) $ acts on elementary tensors 
			as follows:
			\begin{align*}
				&\Big( \big[ \1_{(-\infty,\vartheta)}(A_0) 
				- \1_{(-\infty,\vartheta)}(A_1) \big](\psi \otimes \varphi)
				\Big)(t)	\\
				&= \frac{2}{\pi} \int_{\mathbb{R}_{+}}
				\psi(\tau)
				\1_{[0,\alpha(\vartheta))}(L) 
				\frac{\sin \! \big( (\alpha(\vartheta) - L)^{1/2} (t+\tau) \big)}{t+\tau}
				\varphi
				\dd \tau
			\end{align*}
			for all $ \psi \in \mathsf C_{\mathrm c}(\mathbb{R}_{+}) $ and 
			all $ \varphi \in \mathcal{G} $.
	\end{enumerate}
\end{theorem}
If we represent $ L $ as multiplication operator by the independent variable 
on a von Neumann direct integral (see below), then a scaling transformation
yields the following beautiful representation with separated variables 
for the resolvent difference $A_1-A_0$: 
\begin{theorem}	\label{Resolventendifferenz_unitaer_aequivalent}
	The resolvent difference $A_1-A_0$ is unitarily equivalent to
	\begin{align*}
		\Bigg(
		\bigg[ \Big( {-} \frac{\dd^{2}}{\dd t^{2}} \Big)
		^{\! \mathrm{N}} + 1 \bigg]^{-1} 
		- \bigg[ \Big( {-}\frac{\dd^{2}}{\dd t^{2}} \Big)
		^{\! \mathrm{D}} + 1 \bigg]^{-1} 
		\Bigg)
		\otimes
		(L+1)^{-1}
		\quad
		\text{on } \Lsqr{\R_+} \otimes \mathcal{G}.
	\end{align*}
\end{theorem}
For brevity let us write $ \sigma = \sigma(L) $ for the spectrum of
$L$.  It is well known that $L$ is unitarily equivalent to the
multiplication operator by the independent variable on a von Neumann
direct integral $\int_{\sigma}^{\oplus} \mathcal{G}(\lambda) \dd
\mu(\lambda)$, see Theorem~\ref{Theorem_von_Neumann_direct_integral}
below.  Moreover, from Krein's example~\cite{Krein} we know that the
first factor (the difference of the Neumann and Dirichlet resolvent)
in the previous theorem is a rank $1$ operator with eigenvalue $0$ of
infinite multiplicity and simple eigenvalue $1/2$.  Hence we conclude:
\begin{corollary}
  \label{corollary_from_main_thm__spectral.properties}
  One has
  \begin{align*}
    \sigma(A_1-A_0)
    = \{ 0 \} \cup
    \Big\{ \frac{1}{2(\lambda+1)} : \lambda \in \sigma \Big\},
  \end{align*}
  and the spectral decomposition of $A_1-A_0$ is as follows:
  \begin{enumerate}
  \item $0$ is an eigenvalue of infinite multiplicity.
  \item For $ {\bullet} \in \{ \mathrm{p}, \mathrm{ac},
    \mathrm{sc} \} $ one has 
		$ \sigma_{{\bullet}}(A_1-A_0) \setminus \{ 0 \}
    = \big\{ \frac{1}{2(\lambda+1)} 
    : \lambda \in \sigma_{{\bullet}} \big\} $, and the 
    multiplicity of $ \frac{1}{2(\lambda+1)} $ 
    (with respect to $A_1-A_0$) 
    coincides with the multiplicity of 
    $ \lambda $ (with respect to $ L $) 
    for	$ \dd \mu_{{\bullet}} $-almost all $ \lambda $.
  \end{enumerate}
	In particular, $A_1-A_0$ is compact if and only if 
	$L$ has a purely discrete spectrum.\footnote{This is equivalent with 
	$ (L+1)^{-1} $ being compact, cf.\,\cite{Gorbachuk_Kutovoi,
	Gorbachuk_Kutovoi_II}.}
\end{corollary}
The spectral decomposition of the difference of the spectral 
projections looks as follows:
\begin{theorem}	\label{Ergebnis_I_Differenz_der_Spektralprojektoren}
  Let $0<\vartheta<1$ and 
	let $\alpha(\vartheta) = \frac{1}{\vartheta} - 1 >0$. 
  Then one has:
  \begin{enumerate}
  \item $ \sigma\big( \1_{(-\infty,\vartheta)}(A_0) 
    - \1_{(-\infty,\vartheta)}(A_1) \big)
    = 
    \begin{cases}
      [-1,1]	&	\text{if } \mu \big(\sigma \cap [0,\alpha(\vartheta)) \big) > 0	\\
      \{ 0 \}	&	\text{if } \mu \big(\sigma \cap [0,\alpha(\vartheta)) \big) = 0.
    \end{cases} $
  \item $ \sigma_{\mathrm{p}}\big( \1_{(-\infty,\vartheta)}(A_0) 
    - \1_{(-\infty,\vartheta)}(A_1)  \big)
    = 
    \begin{cases}
      \emptyset	&	\text{if } \mu(\sigma \cap [\alpha(\vartheta),\infty)) = 0	\\
      \{ 0 \}	&	\text{if } \mu(\sigma \cap [\alpha(\vartheta),\infty)) > 0.
    \end{cases} $
    
    If $ \mu(\sigma \cap [\alpha(\vartheta),\infty)) > 0 $ then 
    the multiplicity of the eigenvalue $ 0 $ is infinite.
  \item $ \sigma_{\mathrm{ac}} \big( \1_{(-\infty,\vartheta)}(A_0) 
    - \1_{(-\infty,\vartheta)}(A_1)  \big)
    = 
    \begin{cases}
      [-1,1]	&	\text{if } \mu \big(\sigma \cap [0,\alpha(\vartheta)) \big) > 0	\\
      \emptyset	&	\text{if } \mu \big(\sigma \cap [0,\alpha(\vartheta)) \big) = 0.
    \end{cases} $
    
    If $ \mu \big(\sigma \cap [0,\alpha(\vartheta)) \big) > 0 $ 
		then the (uniform) multiplicity of the absolutely continuous spectrum 
		equals the dimension of 
		$ \int_{\sigma \cap [0, \alpha(\vartheta))}^{\oplus}
    \mathcal{G}(\lambda) \dd\mu(\lambda) $.
  \item The singular continuous spectrum is empty. 
  \end{enumerate}
\end{theorem}
Let us close this subsection with a remark and an example.
\begin{remark}[Link to Hankel operators]
	Observe that
	$\1_{(-\infty,\vartheta)}(A_0) - \1_{(-\infty,\vartheta)}(A_1)$
	is unitarily equivalent to its negative, that its kernel is 
	either trivial or infinite dimensional, and that zero belongs 
	to its spectrum, for all $0<\vartheta<1$. 
	Consequently, the characterisation theorem of bounded self-adjoint 
	Hankel operators~\cite[Theorem 1]{Megretskii_et_al} implies that 
	$\1_{(-\infty,\vartheta)}(A_0) - \1_{(-\infty,\vartheta)}(A_1)$
	is always unitarily equivalent to a 
	Hankel integral operator on $\Lp[2]{\R_{+}}$.
\end{remark}
\begin{example}[Classical half-space]
	If $ L $ is the free Laplacian on $ \mathbb{R}^{n-1} $ 
	for some $ n \geq 2 $ then the difference of the spectral projections 
	associated with $ (-\infty, \vartheta) $ has infinite 
	dimensional kernel, and its
	(absolutely continuous) spectrum equals $ [-1,1] $ 
	and is of infinite multiplicity, 
	for all $0<\vartheta<1$.
\end{example}

%-----------------------------------------------------------------------
\subsection{Structure of the article}
%-----------------------------------------------------------------------

In Section~\ref{sec:tools} we briefly present the main tool of our
analysis, namely the concept of boundary pairs, some facts on the
tensor product of operators, and the von Neumann direct integral
decomposition of a self-adjoint operator.  In
Section~\ref{sec:bd2.half-space} we apply the theory of boundary pairs
to our example and calculate the related objects explicitly.  In
particular, we establish
Theorem~\ref{thm:main1}~\eqref{thm:main1__Teil(a)}.
Section~\ref{sec:main.thm.1} contains the proof of
Theorem~\ref{Resolventendifferenz_unitaer_aequivalent}.  In
Section~\ref{sec:main.thm.2} we establish
Theorem~\ref{thm:main1}~\eqref{thm:main1__Teil(b)} and
Theorem~\ref{Ergebnis_I_Differenz_der_Spektralprojektoren}.

%-----------------------------------------------------------------------
\subsection*{Acknowledgements}
%-----------------------------------------------------------------------
CU would like to thank his PhD supervisor, Vadim Kostrykin, for many
fruitful discussions.  OP would like to thank Vadim Kostrykin for the
hospitality at the Johannes Gutenberg University Mainz, where this
project has been started.  The authors would like to thank André
Froehly for the reference of \cite[Theorem
8.34]{Weidmann}. Furthermore, the authors would like to thank Vladimir
Lotoreichik and Konstantin Pankrashkin for the
reference~\cite{Malamud_Neidhardt}.

%-----------------------------------------------------------------------
%
% 222
\section{Tools}
\label{sec:tools}
%
%-----------------------------------------------------------------------

%-----------------------------------------------------------------------
\subsection{Boundary pairs}
\label{subsec:bd2}
%-----------------------------------------------------------------------

Let us briefly explain the concept of boundary pairs which is
basically an abstract version of boundary value problems for elliptic
operators defined via their quadratic forms.  Details can be found
in~\cite{Post_I}.

Let $\HS$ be a Hilbert space and $\qf h$ a closed and densely defined
quadratic form with domain $\HS^1=\dom (\qf h)$ (i.\,e., $\HS^1$ with
its intrinsic norm defined by $\normsqr[\qf h] u = \qf h(u) +
\normsqr[\HS] u$ is complete).

A \emph{boundary pair} $(\Gamma,\HSaux$) associated with $\qf h$ is a
pair given by another Hilbert space $\HSaux$ and a bounded map $\map
\Gamma {\HS^1} \HSaux$ such that the kernel (null space) $\ker
(\Gamma)$ is dense in $\HS$ and such that the range
$\HSaux^{1/2}=\ran (\Gamma)$ is dense in $\HSaux$.

Given a boundary pair $(\Gamma,\HSaux)$ associated with $\qf h$, we
can define the following objects:
\begin{myitemize}
\item the \emph{(abstract) Neumann operator} $H$ as the operator associated
  with the closed form $\qf h$;
\item the \emph{(abstract) Dirichlet operator} $\HDir$ as the operator associated
  with the closed form $\qf h \restr {\ker \Gamma}$;
\item the \emph{space of weak solutions} $\LS^1(z)=\set{h \in \HS^1}{\qf
  h(h,f)=z\iprod h f \; \forall f \in \ker \Gamma=\HS^{1,\Dir}}$;
\item for $z \notin \spec \HDir$, $\HS^1=\HS^{1,\Dir} \dplus \LS^1(z)$
  (direct sum with closed subspaces); in particular, the
  \emph{(Dirichlet) solution operator} $\map{S(z)=(\Gamma
    \restr{\LS^1(z)})^{-1}}{\ran \Gamma=\HSaux^{1/2}}{\LS^1(z)\subset
    \HS^1}$ is defined;
\item for $z \notin \spec \HDir$, the \emph{Dirichlet-to-Neumann
    (sesquilinear) form} $\qf l_z$ is defined via $\qf
  l_z(\phi,\psi)=(\qf h-z\qf 1)(S(z)\phi,S(-1)\psi)$, $\phi,\psi \in
  \HSaux^{1/2}$;
\item we endow $\HSaux^{1/2}$ with the norm given by
  $\normsqr[\HSaux^{1/2}] \phi = \qf l_{-1}(\phi)=\normsqr[\qf h]{S
  \phi}$.
\end{myitemize}
We say that a boundary pair $(\Gamma,\HSaux)$ associated with $\qf h$
is \emph{elliptically regular} if the associated Dirichlet solution
operator $\map{S=S(-1)}{\HSaux^{1/2}}{\HS^1}$ extends to a bounded
operator $\map {\bar S} \HSaux \HS$, or equivalently, if there exists
a constant $c>0$ such that $\norm[\HS]{S \phi} \le c \norm[\HSaux]
\phi$ for all $\phi \in \HSaux^{1/2}$.  We call $\bar S$ the
\emph{extended solution operator.}  For an elliptic boundary pair, the
Dirichlet-to-Neumann form $\qf l_z$ is sectorial, and the associated
operator, the \emph{Dirichlet-to-Neumann operator} $\Lambda(z)$ has
domain \emph{independent} of $z$.

The main example is the following: let $X$ be an open subset of $\R^n$
with smooth boundary $Y=\bd X$.  Let $\HS=\Lsqr X$, $\qf
h(u)=\int_X\abssqr{\nabla u(x)} \dd x$, $\dom(\qf h)=\Sob X$.
Moreover, let $\Gamma u= u \restr Y$, i.\,e., $\Gamma$ is the (Sobolev)
trace map.  Under suitable conditions (e.g.\ $Y$ is compact or some
curvature assumptions of $Y$), $\map \Gamma {\Sob X} {\Lsqr Y}$ is
bounded, where we consider $Y$ as Riemannian manifold with its natural
$(n-1)$-dimensional measure.  In our example above we have $X=\R_+^n$
and $Y=\{0\} \times \R^{n-1}$.  Then $H$ resp.\ $\HDir$ is the
Neumann resp.\ Dirichlet Laplacian; $\LS^1(z)$ the space of weak
solutions of $(-\Delta-z) h=0$ with $h \in \Sob X$; $S(z)$ is the
solution operator, associating to $\phi \in \ran(\Gamma)$ the weak
solution $h$ with $\Gamma h=\phi$.  Moreover, $\Lambda(z)$ is the
classical Dirichlet-to-Neumann operator, associating to a boundary
function $\map \phi Y \C$ the normal derivative of the function $h \in
\LS^1(z)$ with $\Gamma h=\phi$.

For elliptic boundary pairs, we have the following Krein-type formula
\begin{equation*}
%\label{eq:krein}
  \RNeu(z) - \RDir(z)
  = \bar{S}(z) \Lambda(z)^{-1} \bar{S}(\conj z)^*,
\end{equation*}
(see~\cite[Thm.~4.2~(ii)]{Post_I}), where $\RNeu(z)=(\HNeu-z)^{-1}$
and $\RDir(z)=(\HDir-z)^{-1}$ are the resolvents of the Neumann resp.\
Dirichlet operator.

%-----------------------------------------------------------------------
\subsection{Tensor product of operators}
\label{subsec:tensor}
%-----------------------------------------------------------------------

In this subsection we fix some notation and briefly discuss how 
a result from~\cite{Schmuedgen} about cores for certain 
self-adjoint product type operators carries over 
to the forms associated with these operators; furthermore, 
we present three facts on operators of this product type. 

Let $ T_{k} \geq 0 $ be a self-adjoint operator on a 
complex  Hilbert space $ \mathcal{H}_{k} $ with 
domain $ \dom(T_{k}) $, where $ k=1,2 $. 
We write $ \mathcal{H}_{1} \otimes \mathcal{H}_{2} $ for the 
usual Hilbert space tensor product and 
$ \mathcal{H}_{1} \odot \mathcal{H}_{2} $ for the 
algebraic tensor product of 
$ \mathcal{H}_{1} $ and $ \mathcal{H}_{2} $.

Let $ T \in \{ T_{1}, T_{2} \} $.  Recall
(see~\cite[p.\,145]{Schmuedgen}) that a vector $ f \in
\bigcap_{m=1}^{\infty} \dom(T^{m}) $ is called \emph{bounded for $T$}
if there exists a constant $ B_f > 0 $ such that $ \| T^{m} f \|
\leq B_f^{m} $ for every $ m \in \mathbb{N} $.  In this case we
write $ f \in \mathcal D^{\mathrm b}(T) $.

It follows from~\cite[Theorem 7.23]{Schmuedgen} 
and~\cite[Exercise 17.a]{Schmuedgen}
that the operator $ T_{1} \otimes \id + \id \otimes T_{2} $ is
self-adjoint and that the subspace
\begin{align}	\label{Definition_des_Cores_Db__abstrakt}
	\mathcal D_{\mathrm b} 
	= \Span 
	\{ f_1 \otimes f_2 : f_1 \in \mathcal D^{\mathrm b}(T_{1}), f_2 \in \mathcal D^{\mathrm b}(T_{2}) \}
\end{align}
of $ \mathcal{H}_{1} \otimes \mathcal{H}_{2} $ is an invariant core.
Since $ T_{1} \geq 0 $ and $ T_{2} \geq 0 $ we have
$ T_{1} \otimes \id + \id \otimes T_{2} \geq 0 $.
The following proposition is an easy consequence 
of~\cite[Theorem 7.23]{Schmuedgen} and will be 
very useful.
\begin{proposition}
  \label{prp:core.db}
	The subspace $ \mathcal D_{\mathrm b} $ of $ \mathcal{H}_{1} \otimes \mathcal{H}_{2} $
	defined in~\eqref{Definition_des_Cores_Db__abstrakt}
	is a core for the form associated with the self-adjoint operator
	$ T_{1} \otimes \id + \id \otimes T_{2} $.
\end{proposition}
\begin{proof}
	For brevity let us write 
	$ H = T_{1} \otimes \id + \id \otimes T_{2} $
	and $ \mathcal{H} = \mathcal{H}_{1} \otimes \mathcal{H}_{2} $. 
	It suffices to show that $ \mathcal D_{\mathrm b} $ is a core for the self-adjoint 
	operator $ H^{1/2} $,
	see~\cite[Proposition 10.5]{Schmuedgen}.
	
	It is well known (see, e.\,g.~\cite[Corollary 4.14]{Schmuedgen}) 
	that the domain of $ H $ is a core for $ H^{1/2} $.
	Let $ x \in \dom(H) $. Since $ \mathcal D_{\mathrm b} $ is a core for $ H $ 
	we can choose a sequence 
	$ (x_{m}) \subset \mathcal D_{\mathrm b} $ such that 
	$ x_{m} \rightarrow x $ in $ \mathcal{H} $
	and $ H x_{m} \rightarrow H x $ 
	in $ \mathcal{H} $ as $ m \rightarrow \infty $. 
	It follows directly from the functional calculus for 
	self-adjoint operators and the obvious 
	inequality $ \lambda \leq 1 + \lambda^{2} $ 
	for all $ \lambda \in \mathbb{R} $ that 
	$ H^{1/2} x_{m} \rightarrow H^{1/2} x $ 
	in $ \mathcal{H} $ as $ m \rightarrow \infty $. 
	Consequently $ \mathcal D_{\mathrm b} $ is a core for $ H^{1/2} $, as claimed.
\end{proof}
Here are three more facts on operators of the type 
$ T_{1} \otimes \id + \id \otimes T_{2} $.
\begin{proposition}	\label{Allgemeine_Fakten_zum_Spektrum_von_H}
	Let, as above, $ T_{1} $ and $ T_{2} $ be nonnegative self-adjoint operators.  
	\begin{enumerate}
		\item $ \sigma(T_{1} \otimes \id + \id \otimes T_{2}) 
					= \{ t_{1} + t_{2} : t_{k} \in \sigma(T_{k}), ~ k=1,2 \} $.
		\item \label{Weimann-Formel}
			For all $\alpha \in \R$, all $f_1,g_1 \in \mathcal{H}_1 $, and 
			all $f_2,g_2 \in \mathcal{H}_2 $ one has
			\begin{align*}
				&\langle \1_{(-\infty, \alpha)}(T_1 \otimes \id + \id \otimes T_2) 
				(f_1 \otimes f_2), g_1 \otimes g_2 
				\rangle_{\mathcal{H}_1 \otimes \mathcal{H}_2}	\\
				&= \int_{-\infty}^{\infty} \langle 
				\1_{(-\infty, \alpha - \lambda)}(T_1)f_1, g_1 \rangle_{\mathcal{H}_1}
				\dd \langle \1_{\lambda}(T_2)f_2, g_2 \rangle.
			\end{align*}
			\item The operator $ T_{1} \otimes \id + \id \otimes T_{2} $ has a purely 
					absolutely continuous spectrum if $ T_{1} $ has a purely absolutely 
					continuous spectrum.
	\end{enumerate}
\end{proposition}
\begin{proof}
  Part~(a) follows from~\cite[Corollary 7.25]{Schmuedgen}
  and~\cite[Exc.~18.a]{Schmuedgen}; for Part~(b),
  see~\cite[Thm~8.34]{Weidmann} and for Part~(c)
  see~\cite[Prp~A.2~(iv)]{Malamud_Neidhardt}.
\end{proof}

%-----------------------------------------------------------------------
\subsection{The von Neumann direct integral}
\label{subsec:dir.int}
%-----------------------------------------------------------------------

The theory of von Neumann direct integrals is one of the main tools 
in this paper; for a theoretical background we refer to
\cite[Chapter 7]{Birman_Solomyak}. 
In this subsection we fix some notation and discuss how 
the theory of von Neumann direct integrals can be applied in our example. 

Given a positive finite Borel measure $ \mu $ on $ \mathbb{R} $ we
denote the von Neumann direct integral of separable complex Hilbert
spaces $ \mathcal{G}(\lambda) $ by $ \mathcal{G} =
\int_{\mathbb{R}}^{\oplus} \mathcal{G}(\lambda) \dd \mu(\lambda) $.
Any element $ \phi \in \mathcal{G} $ takes values $ \phi(\lambda) \in
\mathcal{G}(\lambda) $ for $\dd \mu$-almost all $\lambda \in \sigma$.
We will use the notation $\phi = \int_{\mathbb{R}}^{\oplus}
\phi(\lambda) \dd \mu(\lambda)$.  The von Neumann direct integral $
\mathcal{G} $ together with the inner product
\begin{align*}
	\langle \phi_{1}, \phi_{2} \rangle_{\mathcal{G}}
	= \int_{\mathbb{R}} \langle \phi_{1}(\lambda), \phi_{2}(\lambda) \rangle
	_{\mathcal{G}(\lambda)} \dd \mu(\lambda),
	\quad 
	\phi_{1},\phi_{2} \in \mathcal{G},
\end{align*}
is a Hilbert space. The induced norm is denoted by 
$ {\|} {\bullet} {\|}_{\mathcal{G}} $.
We assume without loss of generality that 
$ \mathcal{G}(\lambda) \neq \{ 0 \} $ for $ \dd \mu $-almost 
every $ \lambda $. 
Further we identify the Hilbert spaces 
$ \int_{\mathbb{R}}^{\oplus} \mathcal{G}(\lambda) \dd \mu(\lambda) $ 
and 
$ \int_{\supp(\mu)}^{\oplus} \mathcal{G}(\lambda) 
\dd \mu(\lambda) $, where $ \supp(\mu) $ denotes the support 
of the measure $ \mu $.
We will make use of the following well-known fact:
\begin{theorem}[{\cite[Theorem 1, p.\,177]{Birman_Solomyak}}]
	\label{Theorem_von_Neumann_direct_integral}
	Every self-adjoint operator on a separable complex Hilbert space is 
	unitarily equivalent to the multiplication operator by 
	the independent variable on a von Neumann direct integral.
\end{theorem}
Except for Subsection~\ref{sec:expl.formula.bd2} we will suppose
in Sections~\ref{sec:bd2.half-space}--\ref{sec:main.thm.2}: 
\begin{assumption}
	The operator $ L $ in~\eqref{Gen.Neu/Dir}
	acts by multiplication by the independent variable 
	on a von Neumann direct integral 
	$ \mathcal{G} 
	= \int_{\mathbb{R}}^{\oplus} \mathcal{G}(\lambda) \dd \mu(\lambda) 
	\neq \{ 0 \} $.
\end{assumption}
\begin{remark}
	With this assumption we do not forfeit generality. 
	This is clear in view of 
	Theorem~\ref{Resolventendifferenz_unitaer_aequivalent}, 
	Corollary~\ref{corollary_from_main_thm__spectral.properties}, 
	and Theorem~\ref{Ergebnis_I_Differenz_der_Spektralprojektoren}. 
	In view of Theorem~\ref{thm:main1} we will show
	in Subsections~\ref{sec:expl.formula.bd2}
	and~\ref{subsec:main_thm_1.1} below
	that the corresponding results from 
	Proposition~\ref{Darstellungsformel_fuer_die_Differenz_der_Resolventen} 
	and Lemma~\ref{Spektr.proj.diff.__L_Mult.op.}
	naturally carry over to the situation when $ L $ is not 
	necessarily a multiplication operator.
\end{remark}

%-----------------------------------------------------------------------
%
% 333
\section{The boundary pair of the generalised half-space problem}
\label{sec:bd2.half-space}
%
%-----------------------------------------------------------------------

Let $\mathcal G$ be a (non-trivial) separable Hilbert space and $
\mathcal{H} = \mathsf L_{2}(\mathbb{R}_{+}, \mathcal{G})$.  As
$\mathcal H$ and $ \mathsf L_{2}(\mathbb{R}_{+}) \otimes \mathcal{G} $
are naturally isometrically isomorphic, we will very often identify $
\psi({\bullet}) \phi$ with $ \psi \otimes \phi $ for all $ \psi \in
\Lsqr {\mathbb{R}_{+}} $ and $ \phi \in \mathcal{G} $.
%-----------------------------------------------------------------------
\subsection{The form and its associated operator}
\label{subsec:neu.form}
%-----------------------------------------------------------------------

Let us consider the nonnegative form $ \mathfrak{h} $ on 
%\begin{align*}
$	\mathcal{H}^{1} 
	= \Sob{\R_{+}, \mathcal{G}} 
	\cap \Lp[2]{\R_{+}, \dom(L^{1/2})} $
%\end{align*}
defined by 
\begin{align*}
	\mathfrak{h}(u)
	= \int_{\R_{+}} \bigl( \| u^{\prime}(t) \|_{\mathcal{G}}^{2}
	+ \| L^{1/2} (u(t)) \|_{\mathcal{G}}^{2} \bigr) \dd t,
\end{align*}
where $ \dom(L^{1/2}) $ is equipped with the graph norm of $ L^{1/2} $.
It is easy to see that $ \mathfrak{h} $ is closed. 

Let $H$ be the self-adjoint operator 
\begin{equation*}
  H
	=	\Big( {-} \frac{\dd^{2}}{\dd t^{2}} \Big)^{\! \mathrm{N}}
	\otimes \id 
	+ \id \otimes L
	\quad
	\text{on } 
	\mathsf L_{2}(\mathbb{R}_{+}) \otimes \mathcal{G}.
\end{equation*}
Using the above-mentioned identification of $\mathcal H= \mathsf
L_{2}(\mathbb{R}_{+}, \mathcal{G})$ with $\mathsf
L_{2}(\mathbb{R}_{+}) \otimes \mathcal{G}$, one can show that
\begin{align*}
	\dom(H) 
	= \{ u \in \Sob[2]{\R_{+}, \mathcal{G}} 
	\cap \Lp[2]{\R_{+}, \dom(L)} : u^{\prime}(0) = 0 \},
\end{align*}
see~\cite[Proposition 5.2]{Malamud_Neidhardt}.

\begin{lemma}
  The operator $H$ is associated with the form $ \mathfrak{h} $.
\end{lemma}
\begin{proof}
	For all $ u \in \dom(H) $ and
	all $ v \in \mathcal{H}^{1} $ we have
	\begin{align*}
		\mathfrak{h}(u,v)
		&= \int_{\R_{+}}
		\{ \langle u^{\prime}(t), v^{\prime}(t) \rangle_{\mathcal{G}}
		+ \langle L^{1/2}(u(t)), L^{1/2}(v(t)) \rangle_{\mathcal{G}}
		\} \dd t	\\
		&= \int_{\R_{+}}
		\{ \langle -u^{\prime\prime}(t), v(t) \rangle_{\mathcal{G}}
		+ \langle L(u(t)), v(t) \rangle_{\mathcal{G}}
		\} \dd t
		= \langle Hu, v \rangle_{\mathcal{H}},
	\end{align*}
	where we used integration by parts and the self-adjointness 
	of $ L^{1/2} $.
	Since $ H $ is self-adjoint the claim follows.
\end{proof}
Recall that
\begin{align*}	%\label{Definition_des_Cores_Db}
	\mathcal D_{\mathrm b} 
	= \Span 
	\{ \psi \otimes \phi : 
	\psi \in \mathcal D^{\mathrm b} \big( (- \dd^{2} 
	/ \mathrm{dt}^{2})^\Neu \big), 
	\phi \in \mathcal D^{\mathrm b}(L) \}
	\subset \mathsf L_{2}(\mathbb{R}_{+}) \otimes \mathcal{G}
\end{align*}
is a core for $ H $ as well as for $ \mathfrak{h} $ by
Subsection~\ref{subsec:tensor}.

Functions of the type 
\begin{equation}
  \label{eq:def.h}
  h \colon \R_+ \to \mathcal G,
  \qquad
  t \mapsto h(t) =
  \int_{\sigma}^{\oplus} \exp(\im \sqrt{z-\lambda} \, t)
  \varphi(\lambda) \dd\mu(\lambda) 
\end{equation}
will play an important role in this
paper.   Here, $\sqrt z$ is the square root cut along the positive
half-axis.
First of all we have to check that $ h $ is 
in $ \mathcal{H} $ for all $ z \in \mathbb{C} \setminus [\min \sigma ,\infty) $ 
and all $ \varphi \in \mathcal{G} $.
\begin{lemma} \label{Der_fortgesetzte_Loesungsoperator} Let $ z \in
  \mathbb{C} \setminus [\min \sigma ,\infty) $ and let $ \varphi \in
  \mathcal{G} $. Then the function $h \colon \R_+ \to \mathcal G$
  defined in~\eqref{eq:def.h} is continuous and $ h \in \mathcal{H} $.
\end{lemma}
\begin{proof}
	For every $ t \in \mathbb{R}_{+} $ one has 
	$ \| h(t) \|_{\mathcal{G}} \leq \| \varphi \|_{\mathcal{G}} 
	< \infty $ so $ h $ is $ \mathcal{G} $-valued. 
	By the dominated convergence theorem we see that 
	$ \mathbb{R}_{+} \ni t \mapsto 	h(t) \in \mathcal{G} $ 
	is continuous. 
	Consequently, $ h $ is 
	measurable and we compute
	\begin{align}
                          \label{eq:est.h}
          \| h \|_{\mathcal{H}}^{2}
		&\leq \int_{\mathbb{R}_{+}} \dd t 
		\int_{\sigma} \dd\mu(\lambda) 
		\exp(- 2^{1/2} \, (|z| - \Re(z))^{1/2} \, t) \,
		\| \varphi(\lambda) \|_{\mathcal{G}(\lambda)}^{2}	\\
                \nonumber
		&= \frac{1}{2^{1/2} \, (|z| - \Re(z))^{1/2}} 
		\| \varphi \|_{\mathcal{G}}^{2}	
		< \infty.		\qedhere
	\end{align}
\end{proof}
Next we show:
\begin{lemma}
  \label{Teilmengen_von_dom(h)__II}
  Let $ z \in \mathbb{C}
  \setminus [\min \sigma ,\infty) $ and let $ \varphi \in
  \dom(L^{1/4}) $.  Then the function $h \colon \R_+ \to \mathcal G$
  defined in~\eqref{eq:def.h} is also in $\mathcal{H}^{1}$
\end{lemma}
\begin{proof}
	First consider the case when $ \varphi \in \dom(L) $. 
	By Lemma~\ref{Der_fortgesetzte_Loesungsoperator} we know that 
	$ h \in \mathcal{H} $, and it is straightforward to show that 
	$ h \in \mathcal{H}^{1} $; note that $ h^{\prime} $ exists in 
	the strong sense. 
	
	Now consider the case when $ \varphi \in \dom(L^{1/4}) $. 
	Again, Lemma~\ref{Der_fortgesetzte_Loesungsoperator} shows that 
	$ h $ is in $ \mathcal{H} $. 
	Since $ \dom(L) $ is a core for $ L^{1/4} $ we can approximate 
	$ \varphi $ by a sequence 
	$ (\varphi_{m}) \subset \dom(L) $ with respect to the graph norm 
	of $ L^{1/4} $. 
	Straightforward computations show that 
	$ \| h - h_{m} \|_{\mathcal{H}}
	\xrightarrow{m \rightarrow \infty} 0 $ and 
	$	\mathfrak{h}(h_{k} - h_{m}) \xrightarrow{k,m \rightarrow \infty} 0 $,
	where 
	$ h_{m} = \int_{\sigma}^{\oplus} \exp(\im \sqrt{z-\lambda} \, {\bullet})
	\varphi_{m}(\lambda) \dd\mu(\lambda) $ for all $ m \in \mathbb{N} $.
	Consequently, the closedness of $ \mathfrak{h} $ yields:
	\begin{align}	\label{Teilmengen_von_dom(h)__III}
		h \in \mathcal{H}^{1} 
		\quad	\text{and}	\quad
		\| h-h_m \|_{\mathfrak{h}} \xrightarrow{m \rightarrow \infty} 0.
	\end{align}
	
	This completes the proof of the lemma.
\end{proof}

%-----------------------------------------------------------------------
\subsection{The boundary operator}
\label{subsec:bd.op}
%-----------------------------------------------------------------------

As boundary operator we will choose the restriction to $
\mathcal{H}^{1} $ of the usual boundary operator on the Sobolev space
$ \Sob{\R_{+}, \mathcal{G}} $ that evaluates a given function at zero,
i.\,e., we define the \emph{boundary operator} $ \Gamma \colon
\mathcal{H}^{1} \rightarrow \mathcal{G} $ by $ \Gamma u = u(0) $.
\begin{lemma}	\label{Der_Randoperator_auf_D}
	One has $ \| \Gamma \| \leq 2 $.
\end{lemma}
\begin{proof}
	Let $ u \in \mathcal{H}^{1} $. 
	Define the Lipschitz continuous function 
	$ \chi \colon [0,\infty) \rightarrow [0,1] $ by
	\begin{equation*}
          \chi(t) = 1-t \text{ if } 0 \le t < 1
          \quadtext{and}
          \chi(t) = 0 \text{ if } t \ge 1.
	\end{equation*}
	Then one has
	\begin{align*}
		u(0) 
		= - \big( (\chi \cdot u)(1) - (\chi \cdot u)(0) \big)
		= - \int_{0}^1 (\chi \cdot u)^{\prime}(t) \dd t
		= - \int_{0}^1 \big( \chi^{\prime}(t) \cdot u(t)
		+ \chi(t) \cdot u^{\prime}(t) \big) \dd t.
	\end{align*}
	The result now follows from
	\begin{align*}
		\| \Gamma u \|_{\mathcal{G}}^{2}
		&\leq 
		2 \int_{0}^1 
		\| \chi^{\prime}(t) \cdot u(t) + \chi(t) \cdot u^{\prime}(t)
		\|_{\mathcal{G}}^{2} \dd t	\\
		&\leq
		4 \int_{0}^1 
		\big( \| u(t) \|_{\mathcal{G}}^{2} 
		+ \| u^{\prime}(t) \|_{\mathcal{G}}^{2} \big)
		\dd t
		\\
		&\leq 4 \| u \|_{\mathfrak{h}}^{2}. \qedhere
	\end{align*}
\end{proof}
The proof of the following lemma is straightforward:
\begin{lemma}	\label{Dichtheit insbesondere von ran Gamma}
	The kernel of $ \Gamma $ is dense in $ \mathcal{H} $ 
	with respect to the norm $ \| {\bullet} \|_{\mathcal{H}} $, and
	the range of $ \Gamma $ is dense in $ \mathcal{G} $.
\end{lemma}

Next we define the form
$ \mathfrak{h}^{\Dir} = \mathfrak{h} \restr{\mathcal{H}^{1,\Dir}} $ 
on the closed subspace $ \mathcal{H}^{1,\Dir} = \ker(\Gamma) $ of
$ \mathcal{H}^{1} $. 
Then $ \mathfrak{h}^{\Dir} $ is a densely defined nonnegative closed form. 
We call $ \HDir $, the self-adjoint operator 
associated with $ \mathfrak{h}^{\Dir} $, 
the \emph{Dirichlet operator}. 
We shall show that the Dirichlet operator coincides 
with the self-adjoint operator 
\begin{equation*}
	\Big( {-} \frac{\dd^{2}}{\dd t^{2}} \Big)^{\! \mathrm{D}}
	\otimes \id 
	+ \id \otimes L
	\quad
	\text{on } 
	\mathsf L_{2}(\mathbb{R}_{+}) \otimes \mathcal{G}.
\end{equation*}
We know (see Subsection~\ref{subsec:tensor} above) that
\begin{align*}
	\mathcal D^\Dir_{\mathrm b} 
	= \Span 
	\{ \psi \otimes \phi : 
	\psi \in \mathcal D^{\mathrm b} \big( (- \dd^{2} 
	/ \mathrm{dt}^{2})^\Dir \big), 
	\phi \in \mathcal D^{\mathrm b}(L) \}
	\subset \mathsf L_{2}(\mathbb{R}_{+}) \otimes \mathcal{G}
\end{align*}
is an invariant core for $ \big( {-} \frac{\dd^{2}}
{\dd t^{2}} \big)^{\! \mathrm{D}} \otimes \id 
+ \id \otimes L $.
Note that $ \mathcal D^\Dir_{\mathrm b} \subset \ker(\Gamma) $.
\begin{lemma}
  The Dirichlet operator is given by $ \HDir = \big( {-} \frac{\dd^{2}}
	{\dd t^{2}} \big)^{\! \mathrm{D}} \otimes \id 
	+ \id \otimes L $.
\end{lemma}
\begin{proof}
	For brevity we shall write $ \wtHDir 
	= \big( {-} \frac{\dd^{2}}{\dd t^{2}} \big)^{\! \mathrm{D}} 
	\otimes \id 
	+ \id \otimes L $.
	We will show that $ \wtHDir $ is associated with 
	$ \mathfrak{h}^{\mathrm{D}} $. This is proven in three steps:
      \paragraph{Step 1.}
      Integration by parts yields $ \mathfrak{h}^{\mathrm{D}}(u,f) =
      \big\langle \wtHDir u, f \big\rangle_{\mathcal{H}} $
      for all $ u,f \in \mathcal D^\Dir_{\mathrm b} $.
      
      \paragraph{Step 2.}
			Let $ u \in \mathcal D^\Dir_{\mathrm b} $ and
			let $ \tilde{f} \in \ker(\Gamma) $. 
			Choose $ (f_{k}) \subset \mathcal D_{\mathrm b} $ with
			$ \| \tilde{f} - f_{k} \|_{\mathfrak{h}} 
			\xrightarrow[k \rightarrow \infty]{} 0 $.
			Integration by parts yields $ \mathfrak{h}(u, f_{k})
			= \big\langle \wtHDir u, f_{k} \big\rangle_{\mathcal{H}}
			- \langle \Gamma(u^{\prime}), \Gamma f_{k} \rangle_{\mathcal{G}} $,
			where $ \Gamma(u^{\prime}) \in \dom(L) \subset \mathcal{G} $.
			As $ k $ tends to infinity we obtain that 
			$ \mathfrak{h}(u, f_{k}) \rightarrow \mathfrak{h}(u, \tilde{f}) 
			= \mathfrak{h}^{\mathrm{D}}(u, \tilde{f}) $ and, 
			on the other hand, 
			$ \big\langle \wtHDir u, f_{k} \big\rangle_{\mathcal{H}}
			- \langle \Gamma(u^{\prime}), \Gamma f_{k} \rangle_{\mathcal{G}}
			\rightarrow 
			\big\langle \wtHDir u, \tilde{f} \big\rangle_{\mathcal{H}}
			- \langle \Gamma(u^{\prime}), \Gamma \tilde{f} \rangle_{\mathcal{G}}
			= \big\langle \wtHDir u, \tilde{f} \big\rangle_{\mathcal{H}} $.
		\paragraph{Step 3.}
			Let $ \tilde u \in \dom \! \big( \wtHDir \big) $ and
			let $ \tilde{f} \in \ker(\Gamma) $. 
			Choose $ (u_m) \subset \mathcal D^\Dir_{\mathrm b} $ with
			$ \| \tilde u - u_m \|_{\wtHDir} 
			\xrightarrow[m \rightarrow \infty]{} 0 $.
			Then, by Step 1 and the positivity of $ \wtHDir $, one has
			\begin{align*}
				\mathfrak{h}^{\mathrm{D}}(u_k - u_m)
				= \big| \big\langle \wtHDir (u_k - u_m), 
				u_k - u_m \big\rangle_{\mathcal{H}} \big|
				\leq \| u_k - u_m \|_{\wtHDir}^{2}
				\quad	\text{for all } k,m \in \mathbb{N}
			\end{align*}
			so $ (u_m)_m $ is Cauchy with respect to 
			$ \| {\bullet} \|_{\mathfrak{h}^{\mathrm{D}}} $.
			Since $ \mathfrak{h}^{\mathrm{D}} $ is closed it follows that 
			$ \tilde u \in \ker(\Gamma) $ and 
			$ \| \tilde u - u_m \|_{\mathfrak{h}^{\mathrm{D}}} 
			\xrightarrow[m \rightarrow \infty]{} 0 $.
			As $ m $ tends to infinity we obtain that 
			$ \mathfrak{h}^{\mathrm{D}}(u_m, \tilde{f}) 
			\rightarrow \mathfrak{h}^{\mathrm{D}}(\tilde u, \tilde{f}) $
			and, on the other hand, 
			$ \big\langle \wtHDir u_m, 
			\tilde{f} \big\rangle_{\mathcal{H}}
			\rightarrow 
			\big\langle \wtHDir \tilde u, 
			\tilde{f} \big\rangle_{\mathcal{H}} $.
			Consequently, 
			\begin{align*}
				\mathfrak{h}^{\mathrm{D}}(\tilde u, \tilde{f})
				= \big\langle \wtHDir \tilde u, 
				\tilde{f} \big\rangle_{\mathcal{H}} 
			\end{align*}
			and thus $ \wtHDir \subset \HDir $.
                        Since $ \wtHDir $ and $ \HDir $ are
                        both self-adjoint we conclude that $
                        \wtHDir = \HDir $.
\end{proof}
\begin{lemma}
  \indent
	\begin{enumerate}
		\item The operators $ H $ and $ \HDir $ 
					are unitarily equivalent.
		\item The spectrum of $ H $ is 
					purely absolutely continuous filling in 
					the interval $ [\min \sigma, \infty) $; 
					the same is true for $ \HDir $.
	\end{enumerate}
\end{lemma}
\begin{proof}
  (a)~Since the Neumann and Dirichlet Laplacians on $
  L_{2}(\mathbb{R}_{+}) $ are unitarily equivalent it follows that $ H
  $ and $ \HDir $ are also unitarily equivalent.  Part~(b) follows
  from Proposition~\ref{Allgemeine_Fakten_zum_Spektrum_von_H}.
\end{proof}
\begin{remark}
  One can actually show that the domain of $ \HDir $ is given by
	\begin{align*}
		\dom(H^{\mathrm{D}}) 
		= \{ u \in \Sob[2]{\R_{+}, \mathcal{G}} 
		\cap \Lp[2]{\R_{+}, \dom(L)} : u(0) = 0 \},
	\end{align*}
	see~\cite[Proposition 5.2]{Malamud_Neidhardt}.
\end{remark}

%-----------------------------------------------------------------------
\subsection{The solution operator and the range of the boundary
  operator}
\label{subsec:sol.op}
%-----------------------------------------------------------------------

Let $ z \in \mathbb{C} \setminus [\min \sigma, \infty) $. Define 
\begin{align*}
	\mathcal{N}^{1}(z) 
	= \{ h \in \mathcal{H}^{1} : 
	\mathfrak{h}(h,f) = z \langle h,f \rangle_{\mathcal{H}} 
	\text{ for all } f \in \ker (\Gamma) \}.
\end{align*}
The so-called \emph{solution operator}, given formally by $ S(z) =
\left( \Gamma \restr{\mathcal{N}^{1}(z)} \right)^{-1}$, associates to
a boundary value $ \varphi \in \ran(\Gamma) $ the unique element $ h
\in \mathcal{N}^{1}(z)$ such that $ \Gamma h = \varphi $
(see~\cite[Prp~2.9]{Post_I}).
\begin{lemma}	\label{Loesungsoperator_I}
	One has 
	$ \dom(L^{1/4}) \subset \ran(\Gamma) $ and,
	for every $ z \in \mathbb{C} \setminus [\min \sigma, \infty) $,
	\begin{align}	\label{Formel_Loesungsoperator_I}
		S(z) \restr{\dom(L^{1/4})} \varphi 
		= \int_{\sigma}^{\oplus} 
	\exp \! \big( \im \sqrt{z - \lambda} \, {\bullet}  \big) 
	\varphi(\lambda) \dd\mu(\lambda). 
	\end{align}
\end{lemma}
\begin{proof}
  The lemma is proven in two steps. First we show that $ \dom(L)
  \subset \ran(\Gamma) $ and~\eqref{Formel_Loesungsoperator_I} holds
  on $ \dom(L) $.  Then, by approximation, we obtain that $
  \dom(L^{1/4}) \subset \ran(\Gamma) $ and~\eqref{Formel_Loesungsoperator_I} holds on $ \dom(L^{1/4}) $.
      \paragraph{Step 1.}
      Let $ \varphi \in \dom(L) $ and let $ h = \int_{\sigma}^{\oplus}
      \exp \big( \im \sqrt{z - \lambda} \, \bullet \big)
      \varphi(\lambda) \dd\mu(\lambda) $.  By
      Lemma~\ref{Teilmengen_von_dom(h)__II} we know that $ h \in
      \mathcal{H}^{1} $ and hence $ \Gamma h = \varphi $.  It remains
      to show that $ h \in \mathcal{N}^{1}(z) $.  This is proven as
      follows:
			
      Let $\Phi \in \mathcal D_{\mathrm b} $.  A straightforward computation
      shows that
      \begin{align*}
        \mathfrak{h}(h, \Phi) 
        &= \langle h^{\prime}, \Phi^{\prime} 
        \rangle_{\mathcal{H}}
        + \int_{\mathbb{R}_{+}} 
        \langle L \big( h(t) \big), \Phi(t) \rangle_{\mathcal{G}} \dd t	\\
        &= z \langle h, \Phi \rangle_{\mathcal{H}} 
        - \im \, \int_{\sigma} 
        \langle \sqrt{z - \lambda} \, \varphi(\lambda), 
        (\Gamma \Phi)(\lambda) \rangle_{\mathcal{G}(\lambda)}
        \dd\mu(\lambda).
      \end{align*}
      Now let $ f \in \ker(\Gamma) $. Choose a sequence 
      $ (\Phi_{m}) \subset \mathcal D_{\mathrm b} $ with 
      $ \| f - \Phi_{m} \|_{\mathfrak{h}} 
      \xrightarrow{m \rightarrow \infty} 0 $. 
      Clearly $ \mathfrak{h}(h, \Phi_{m}) 
      \xrightarrow{m \rightarrow \infty} 
      \mathfrak{h}(h, f) $ and 
      $ z \, \langle h, \Phi_{m} \rangle_{\mathcal{H}} 
      \xrightarrow{m \rightarrow \infty} 
      z \, \langle h, f \rangle_{\mathcal{H}} $, 
      and an easy computation shows that
      $| {-} \im \, \int_{\sigma} 
      \langle \sqrt{z - \lambda} \, \varphi(\lambda), 
      (\Gamma \Phi_{m})(\lambda) \rangle_{\mathcal{G}(\lambda)} 
      \dd\mu(\lambda)|	 
      \xrightarrow{m \rightarrow \infty} 0 $. 
      It follows that $ h $ is in $ \mathcal{N}^{1}(z) $. 

    \paragraph{Step 2.}
    Let $ \varphi \in \dom(L^{1/4}) $ and let
		$ h = \int_{\sigma}^{\oplus} 
		\exp \big( \im \sqrt{z - \lambda} \, \bullet \big) 
		\varphi(\lambda) \dd\mu(\lambda) $. 
		Again, we know by Lemma~\ref{Teilmengen_von_dom(h)__II} 
		that $ h \in \mathcal{H}^{1} $ and hence 
		$ \Gamma h = \varphi $. 
			
		Now choose a sequence 
		$ (\varphi_{m}) \subset \dom(L) $ with 
		$ \| \varphi - \varphi_{m} \|_{L^{1/4}} 
		\xrightarrow{m \rightarrow \infty} 0 $. 
		By Step 1 we know that   
		$ h_{m} = \int_{\sigma}^{\oplus} 
		\exp \big( \im \sqrt{z - \lambda} \, \bullet \big) 
		\varphi_{m}(\lambda) \dd\mu(\lambda)
		\in \mathcal{N}^{1}(z) $ 
		for all $ m \in \mathbb{N} $, and~\eqref{Teilmengen_von_dom(h)__III} implies that 
		$ \| h - h_{m} \|_{\mathfrak{h}} 
		\xrightarrow{m \rightarrow \infty} 0 $. 
		Consequently, $ h \in \mathcal{N}^{1}(z) $.  
		This completes the proof of the lemma.
\end{proof}
The following proposition shows that 
$ \ran(\Gamma) \subset \dom(L^{1/4}) $ so, in fact, 
$ S(z) \restr{\dom(L^{1/4})} = S(z) $.
\begin{proposition}	\label{Berechnung von ran Gamma}
	One has $ \ran(\Gamma) \subset \dom(L^{1/4}) $.
\end{proposition}
\begin{proof}
	We decompose $ \mathcal{H}^{1} $ into the orthogonal sum of 
	$ \mathcal{N}^{1} = \mathcal{N}^{1}(-1) $ and 
	$ \ker(\Gamma) $. 
	Since $ \Gamma $ is linear it suffices to show that 
	$ \Gamma h \in \dom(L^{1/4}) $ for all $ h \in \mathcal{N}^{1} $.
	This is proven in four steps:		
		\paragraph{Step 1.} 
			Let $ h \in \mathcal{N}^{1} $. Choose a sequence 
			$ (\tilde{h}_{m}) \subset \mathcal D_{\mathrm b} $ with 
			$ \| h - \tilde{h}_{m} \|_{\mathfrak{h}} 
			\xrightarrow{m \rightarrow \infty} 0 $. Put
			\begin{equation*}
				h_{m} = P_{\mathcal{N}^{1}} \tilde{h}_{m}, 
				\quad m \in \mathbb{N},
			\end{equation*}
			where $ P_{\mathcal{N}^{1}} $ denotes the orthogonal projection of 
			$ \mathcal{H}^{1} $ onto $ \mathcal{N}^{1} $.
                        \paragraph{Step 2.} Let $ m \in \mathbb{N} $
                        and set $ \varphi_{m} = \Gamma h_{m} $. Then
                        one has:
			\begin{align*}
				\varphi_{m}
				= \Gamma  P_{\mathcal{N}^{1}} \tilde{h}_{m} 
				= \Gamma P_{\mathcal{N}^{1}} \tilde{h}_{m}
				+ \Gamma P_{\ker(\Gamma)} \tilde{h}_{m} 
				= \Gamma \tilde{h}_{m}  
				\in \dom(L),
			\end{align*}
			where $ P_{\ker(\Gamma)} $ denotes the orthogonal projection of 
			$ \mathcal{H}^{1} $ onto $ \ker(\Gamma) $.
			By Lemma~\ref{Loesungsoperator_I} we know that
			\begin{multline*}
				\int_{\sigma}^{\oplus} 
				\exp \big( {-} (1 + \lambda)^{1/2}  \bullet \big) 
				\varphi_{m}(\lambda)
				\dd\mu(\lambda) 
				\in \mathcal{N}^{1}
                                \qquad\text{and}\\
				\Gamma \Big(
				{\int_{\sigma}^{\oplus}} 
				\exp \big( {-} (1 + \lambda)^{1/2} \bullet \big) 
				\varphi_{m}(\lambda)
				\dd\mu(\lambda) 
				\Big)
				= \varphi_{m}.
			\end{multline*}
			Since $\Gamma \restr{\mathcal{N}^{1}} $ is injective we thus obtain:
			\begin{align*}
				h_{m}
				= \int_{\sigma}^{\oplus} 
				\exp \big( {-} (1 + \lambda)^{1/2}  \bullet \big) 
				\varphi_{m}(\lambda)
				\dd\mu(\lambda).
			\end{align*}
	\paragraph{Step 3.} Clearly
			$ \| h - h_{m} \|_{\mathfrak{h}}
			= \| P_{\mathcal{N}^{1}} (h - \tilde{h}_{m}) \|_{\mathfrak{h}}
			\leq \| h - \tilde{h}_{m} \|_{\mathfrak{h}}
			\xrightarrow{m \rightarrow \infty} 0 $.
			It follows that 
			\begin{align*}
				\| \Gamma h - \varphi_{m} \|_{\mathcal{G}}
				= \| \Gamma h - \Gamma h_{m} \|_{\mathcal{G}}
				\xrightarrow[m \rightarrow \infty]{} 0.
			\end{align*}
	\paragraph{Step 4.} 
			We already know that $ (h_{m})_m $ is a Cauchy sequence 
			with respect to $ \| {\bullet} \|_{\mathfrak{h}} $. 
			A straightforward computation shows that
			\begin{align*}
				\| h_{k} - h_{m} \|_{\mathfrak{h}}^{2}
				&\geq 
				\int_{\sigma} 
				\lambda 
				\| \varphi_{k}(\lambda) - 
				\varphi_{m}(\lambda) \|_{\mathcal{G}(\lambda)}^{2}
				\int_{\mathbb{R}_{+}} 
				\exp \big( {-}2 (1+\lambda)^{1/2} \, t \big)	\dd t \dd \mu(\lambda) 
\\
				&=
				\frac{1}{2}
				\int_{\sigma} 
				\Big(\frac{\lambda}{1 + \lambda}\Big)^{1/2}
                                  \cdot \lambda^{1/2}
				\| \varphi_{k}(\lambda) - 
				\varphi_{m}(\lambda) \|_{\mathcal{G}(\lambda)}^{2}
				\dd \mu(\lambda)
			\end{align*}
			for all $ k,m \in \mathbb{N} $.  
			Choose $ \lambda_{0} > 0 $ large enough such that $(\lambda/(1 + \lambda))^{1/2} \geq 1/2$ 
			for all $ \lambda \geq \lambda_{0} $.  
			Then we obtain:
			\begin{align*}
				\| \varphi_{k} - \varphi_{m} \|_{L^{1/4}}^{2}
				\leq
				\Big( 1 + \lambda_{0}^{1/2} \Big) 
				\| \varphi_{k} - \varphi_{m} \|_{\mathcal{G}}^{2}
				+ 4 \| h_{k} - h_{m} \|_{\mathfrak{h}}^{2}
				\xrightarrow{k,m \rightarrow \infty} 0.
			\end{align*}
			Since $ \dom(L^{1/4}) $ is complete with respect to 
			$ \| {\bullet} \|_{L^{1/4}} $ there exists 
			$ \varphi \in \dom(L^{1/4}) $ such that
			$ \| \varphi - \varphi_{m} \|_{L^{1/4}} 
			\xrightarrow{m \rightarrow \infty} 0 $. 
			Consequently, one has $ \Gamma h = \varphi 
			\in \dom(L^{1/4}) $, as claimed.
\end{proof}
\begin{remark}
	$ \Gamma $ is surjective if and only if $ L $ is bounded.
\end{remark}
We have thus computed the solution operator $ S(z) $ at every point $
z \in \mathbb{C} \setminus [\min \sigma ,\infty) $.  In particular, if
$ z = -1 $ and $ \varphi \in \dom(L^{1/4}) $ then~\eqref {eq:est.h}
tells us that
\begin{align*}	%\label{Randpaar_ist_elliptisch_regulaer}
	\| S(-1) \varphi \|_{\mathcal{H}}^{2}
	\leq \frac{1}{2} \| \varphi \|_{\mathcal{G}}^{2}.
\end{align*}
This inequality proves:
\begin{lemma}	\label{Randpaar_ist_elliptisch_regulaer}
	The boundary pair $ (\Gamma, \mathcal{G}) $ is 
	elliptically regular.
\end{lemma}

%-----------------------------------------------------------------------
\subsection{The extended solution operator and its adjoint}
\label{subsec:ell.reg}
%-----------------------------------------------------------------------

Let $ z \in \mathbb{C} \setminus [\min \sigma ,\infty) $.  According
to~\eqref {eq:est.h} we know that
\begin{align*}
	\mathcal{G} \ni \varphi \mapsto
	\int_{\sigma}^{\oplus} \exp(\im \sqrt{z-\lambda} \, \bullet)
	\varphi(\lambda) \dd\mu(\lambda)
	\in \mathcal{H}
\end{align*}
defines a bounded operator.  In the preceding subsection we have shown
that the solution operator $ S(z)\colon \ran(\Gamma) \rightarrow
\mathcal{H}^{1} \subset \mathcal{H} $ is given by $ S(z) \varphi =
\int_{\sigma}^{\oplus} \exp(\im \sqrt{z-\lambda} \, \bullet)
\varphi(\lambda) \dd\mu(\lambda) $.  As, by Lemma~\ref{Dichtheit
  insbesondere von ran Gamma}, $ \ran(\Gamma) $ is dense in $
\mathcal{G} $ we can extend this formula to all of $\mathcal G$:
\begin{lemma}	\label{Der_fortgesetzte_Loesungsoperator__II}
	If $ z \in \mathbb{C} \setminus [\min \sigma ,\infty) $ 
	then the unique bounded extension of $ S(z) $ to 
	$ \mathcal{G} $ is given by 
	\begin{equation*}
		\bar{S}(z)\colon \mathcal{G} \rightarrow \mathcal{H},
		\quad
		\bar{S}(z) \varphi
		= \int_{\sigma}^{\oplus} \exp(\im \sqrt{z-\lambda} \, \bullet)
		\varphi(\lambda) \dd\mu(\lambda).
\end{equation*}
\end{lemma}
Next we compute the adjoint of the extended solution operator.
\begin{lemma}	\label{Der_Adjungierte_des_fortgesetzten_Loesungsoperators}
	If $ z \in \mathbb{C} \setminus [\min \sigma ,\infty) $ 
	then the bounded operator
	$( \bar{S}(\bar{z}))^{\ast}\colon \mathcal{H} 
	\rightarrow \mathcal{G} $ acts on elementary tensors 
	as follows:
	\begin{align*}
		\big( \bar{S}(\bar{z}) \big)^{\ast} (\psi \otimes \eta)
		= \int_{\sigma}^{\oplus} 
		\Bigl(\int_{\mathbb{R}_{+}}
		\psi(t)
		\exp \! \big( \im \sqrt{z - \lambda} \, t \big)
                 \dd t \Bigr) \eta(\lambda)
                 \dd\mu(\lambda) 
	\end{align*}
	for all $ \psi \in \mathsf C_{\mathrm c}(\mathbb{R}_{+}) $ 
	and all $ \eta \in \mathcal{G} $. 
	Consequently, $( \bar{S}(\bar{z}))^{\ast} $ can be 
	evaluated explicitly on the dense subspace
	$ \mathsf C_{\mathrm c}(\mathbb{R}_{+}) \odot \mathcal{G} $ 
	of $ \mathcal{H} $.
\end{lemma}
\begin{proof}
	Standard arguments show that
	\begin{align}	\label{Formel_Adjungierter__Existenz}
		\int_{\sigma}^{\oplus} 
		\Bigl(\int_{\mathbb{R}_{+}}
		\psi(t) 
		\exp \! \big( \im \sqrt{z - \lambda} \, t \big)
                 \dd t \Bigr) \eta(\lambda)
                 \dd \mu(\lambda) 
		\in \mathcal{G}.
	\end{align}
	Let $ \phi \in \mathcal{G} $. 
	By Fubini's theorem,
	\begin{align*}
		\big\langle \big( \bar{S}(\bar{z}) \big)^{\ast} (\psi \otimes \eta), 
		\phi \big\rangle_{\mathcal{G}} 
		&= \big\langle \psi \otimes \eta, 
		\bar{S}(\bar{z}) \phi \big\rangle_{\mathcal{H}}	\\
		&= \int_{\sigma} 
		\int_{\mathbb{R}_{+}}
		\Big\langle 
		\psi(t) \,
		\overline{\exp(\im \sqrt{\bar{z} - \lambda} \, t)}
                \, \eta(\lambda), 
		\phi(\lambda)
		\Big\rangle_{\mathcal{G}(\lambda)}
                 \dd t \dd\mu(\lambda).
	\end{align*}
	It is easily seen that $ \overline{\exp(\im \sqrt{\bar{z} - \lambda} \, t)}
	= \exp(\im \sqrt{z - \lambda} \, t) $.
	Therefore, \eqref{Formel_Adjungierter__Existenz} implies
	\begin{align*}
		\big\langle \big( \bar{S}(\bar{z}) \big)^{\ast} (\psi \otimes \eta), 
		\phi \big\rangle_{\mathcal{G}} 
		= \Big\langle 
		\int_{\sigma}^{\oplus}
                \Bigl(
		\int_{\mathbb{R}_{+}} 
		\psi(t) 
		\exp \! \big( \im \sqrt{z - \lambda} \, t \big)
                 \dd t \Bigr)
		\eta(\lambda)
                \dd\mu(\lambda) , 
		\phi
		\Big\rangle_{\mathcal{G}}.
	\end{align*}
	Since $ \phi \in \mathcal{G} $ was arbitrary this proves the lemma.
\end{proof}

%-----------------------------------------------------------------------
\subsection{The Dirichlet-to-Neumann operator}
\label{subsec:dtn.op}
%-----------------------------------------------------------------------

We can think of the Dirichlet-to-Neumann operator $ \Lambda(z) $ as
follows (see~\cite[top of p.\,1053]{Post_I}): it maps certain
boundary values $ \varphi \in \dom(\Lambda(z)) \subset \dom \big(
L^{1/4} \big) $ to the ``normal'' derivative $ \partial_{\mathrm{n}} h
$ of the corresponding Dirichlet solution $ h $.  In our situation
this means:
\begin{align*}
	\Lambda(z) \varphi
	&= - \left. \frac{\partial}{\partial t} 
	\Big( S(z) \varphi \Big) \right|_{t=0}	\\
	&= - \left.
	\int_{\sigma}^{\oplus} 
	\im \sqrt{z - \lambda} ~
	\exp \! \big( \im \sqrt{z - \lambda} \, t \big) 
	\varphi(\lambda) \dd\mu(\lambda)
	\right|_{t=0}	\\
	&= - \int_{\sigma}^{\oplus} 
	\im \sqrt{z - \lambda} ~
	\varphi(\lambda) \dd\mu(\lambda).
\end{align*}
As we will show in 
Lemma~\ref{Berechnung_des_Dirichlet-zu-Neumann-Operators} below, 
this formal computation indeed gives us the correct result.

Let $ z \in \mathbb{C} \setminus [\min \sigma ,\infty) $.
Define 
$	\mathfrak{l}_{z} 
: \dom \big( L^{1/4} \big) \times \dom \big( L^{1/4} \big)
\rightarrow \mathbb{C} $
by
\begin{align*}
	\mathfrak{l}_{z}(\varphi, \eta)
	= \mathfrak{h} \big( S(z)\varphi, S(-1)\eta \big)
	- z \langle S(z)\varphi, S(-1)\eta \rangle_{\mathcal{H}}.
\end{align*}
Then, by~\cite[Theorem 2.12]{Post_I}, $	\mathfrak{l}_{z} $ is a 
bounded form. 
We call $ \mathfrak{l}_{z} $ the \emph{Dirichlet-to-Neumann form}. 
One has:
\begin{lemma}	\label{Die_Dirichlet-zu-Neumann-Form}
	If $ z \in \mathbb{C} \setminus [\min \sigma ,\infty) $ then 
	$ \mathfrak{l}_{z} $ is given by
	\begin{align}	\label{Formel_Dirichlet-zu-Neumann-Form}
		\mathfrak{l}_{z}(\varphi, \eta)
		= \int_{\sigma} 
		\big( {-} \im \sqrt{z - \lambda} \big) 
		\langle \varphi(\lambda), \eta(\lambda) 
		\rangle_{\mathcal{G}(\lambda)}
		\dd\mu(\lambda) 
	\end{align}	
	for all $ \varphi, \eta \in \dom(L^{1/4}) $.
\end{lemma}
\begin{proof}
	The lemma is proven in two steps.
	First we show~\eqref{Formel_Dirichlet-zu-Neumann-Form}
	for $ \varphi, \eta \in \dom(L) $, 
	and then we complete the proof by approximation.
		\paragraph{Step 1.}
			Let $ \varphi, \eta \in \dom(L) $. 
			Using Lemma~\ref{Teilmengen_von_dom(h)__II} and 
			Fubini's theorem we compute:
			\begin{align*}
				\mathfrak{l}_{z}(\varphi, \eta)
				&= \int_{\sigma} 
				\langle \varphi(\lambda), \eta(\lambda) 
				\rangle_{\mathcal{G}(\lambda)} \cdot
                                \\
				&\qquad	\cdot \int_{\mathbb{R}_{+}}
				\exp \Big( \big( \im \sqrt{z - \lambda} 
				- (1+\lambda)^{1/2} \big) t \Big)
                                \dd t
				\left( \im \sqrt{z - \lambda} \, 
				(- (1+\lambda)^{1/2}) + \lambda - z \right)	
                               \dd\mu(\lambda) \\
				&= \int_{\sigma} 
				\big( {-} \im \sqrt{z - \lambda} \big) 
				\langle \varphi(\lambda), \eta(\lambda) 
				\rangle_{\mathcal{G}(\lambda)}
				\dd\mu(\lambda).
			\end{align*}

\paragraph{Step 2.} Let $ \varphi, \eta \in
                        \dom(L^{1/4}) $.  Choose two sequences $
                        (\varphi_{m}) \subset \dom(L) $ and $
                        (\eta_{m}) \subset \dom(L) $ such that $ \|
                        \varphi - \varphi_{m} \|_{L^{1/4}}
                        \xrightarrow{m \rightarrow \infty} 0 $ and $
                        \| \eta - \eta_{m} \|_{L^{1/4}} \xrightarrow{m
                          \rightarrow \infty} 0 $.
                        By~\eqref{Teilmengen_von_dom(h)__III} we know
                        that
                         \begin{equation*}
                            \| S(z) \varphi - S(z)\varphi_{m}
                           \|_{\mathfrak{h}} \xrightarrow{m
                             \rightarrow \infty} 0
                           \quadtext{and}
                           \| S(z)
                           \eta - S(z)\eta_{m} \|_{\mathfrak{h}}
                           \xrightarrow{m \rightarrow \infty} 0.
                         \end{equation*}
			Consequently,
			\begin{multline*}
				\mathfrak{h} \big( S(z)\varphi_{m}, S(-1)\eta_{m} \big)
				- z \langle S(z)\varphi_{m}, S(-1)\eta_{m}
				\rangle_{\mathcal{H}}	\\
				\xrightarrow{m \rightarrow \infty} 
				\mathfrak{h} \big( S(z) \varphi, S(-1) \eta \big)
				- z \langle S(z) \varphi, S(-1) \eta 
				\rangle_{\mathcal{H}}.
			\end{multline*}
			Furthermore a straightforward computation shows that
			\begin{align*}
				\int_{\sigma} 
				\big( {-} \im \sqrt{z - \lambda} \big) 
				\langle \varphi_{m}(\lambda), \eta_{m}(\lambda) 
				\rangle_{\mathcal{G}(\lambda)}
				\dd\mu(\lambda) 
				\xrightarrow{m \rightarrow \infty}
				\int_{\sigma} 
				\big( {-} \im \sqrt{z - \lambda} \big) 
				\langle \varphi(\lambda), \eta(\lambda) 
				\rangle_{\mathcal{G}(\lambda)}
				\dd\mu(\lambda).
			\end{align*}
			Thus~\eqref{Formel_Dirichlet-zu-Neumann-Form}
			holds and the lemma is proven.
\end{proof}
As the boundary pair $ (\Gamma, \mathcal{G}) $ is elliptically
regular, it follows from \cite[Theorem 3.8]{Post_I} that the
Dirichlet-to-Neumann form is closed and sectorial for all $ z \in
\mathbb{C} \setminus [\min \sigma ,\infty)$.  Let $ \Lambda(z) $ be
the closed operator associated with $ \mathfrak{l}_{z} $, i.\,e.,
\begin{align}	\label{Domain_Dirichlet-zu-Neumann-Operator}
	\dom \! \big( \Lambda(z) \big)
	= \big\{ \varphi \in \dom(L^{1/4}) :
	\exists \zeta \in \mathcal{G} \, \forall \eta \in \dom(L^{1/4}), \,
	\mathfrak{l}_{z}(\varphi, \eta) 
	= \langle \zeta, \eta \rangle_{\mathcal{G}} \big\}
\end{align}
and $ \Lambda(z) \varphi = \zeta $.
We call $ \Lambda(z) $ the 
\emph{Dirichlet-to-Neumann operator}. 
One has:
\begin{lemma}	\label{Dirichlet-zu-Neumann-Operators__I}
If $ z \in \mathbb{C} \setminus [\min \sigma ,\infty) $ then 
$ \dom \! \big( \Lambda(z) \big) \supset \dom(L^{1/2}) $ and
	\begin{align*}	
		\Lambda(z) \restr{\dom(L^{1/2})} \varphi
		= \int_{\sigma}^{\oplus} (- \im \, \sqrt{z - \lambda})
		\varphi(\lambda) \dd\mu(\lambda).
	\end{align*}
\end{lemma}
\begin{proof}	
	Let $ \varphi \in \dom(L^{1/2}) $. Then 
	\begin{align*}
		\zeta =
		\int_{\sigma}^{\oplus} (- \im \, \sqrt{z - \lambda})
		\varphi(\lambda) \dd\mu(\lambda) 
		\quad
		\text{is in } \mathcal{G}.
	\end{align*}
	Therefore, Lemma~\ref{Die_Dirichlet-zu-Neumann-Form} implies that 
	$ \mathfrak{l}_{z}(\varphi, \eta) 
	= \langle \zeta, \eta \rangle_{\mathcal{G}} $
	for all $ \eta \in \dom(L^{1/4}) $.
	This proves the lemma.
\end{proof}
Furthermore it follows from~\cite[Theorem 3.8]{Post_I} that
$ \dom \! \big( \Lambda(z) \big) = \dom \! \big( \Lambda(-1) \big) $
is independent of $ z \in \mathbb{C} \setminus [\min \sigma ,\infty) $. 
The next lemma shows that 
$ \dom \! \big( \Lambda(-1) \big) \subset \dom(L^{1/2}) $ so,
in fact, $ \Lambda(z) \restr{\dom(L^{1/2})} = \Lambda(z) $.
\begin{lemma}	\label{Berechnung_des_Dirichlet-zu-Neumann-Operators}
	One has $ \dom \! \big( \Lambda(-1) \big) \subset \dom(L^{1/2}) $.
\end{lemma}
\begin{proof}
	First we observe that for all $ \eta \in \dom(L^{1/4}) $ we have
	\begin{align*}
		\int_{\sigma}^{\oplus} 
		(1 + \lambda)^{-1/4}
		\eta(\lambda) \dd\mu(\lambda) 
		\in \dom(L^{1/4}).
	\end{align*}
	Now let $ \varphi \in \dom \! \big( \Lambda(-1) \big) $. 
	Choose $ \zeta \in \mathcal{G} $ according to~\eqref{Domain_Dirichlet-zu-Neumann-Operator}. 
	Then clearly 
	\begin{align*}
		\int_{\sigma}^{\oplus} 
		(1 + \lambda)^{-1/4}
		\zeta(\lambda) \dd\mu(\lambda) 
		\in \mathcal{G}
	\end{align*}
	and, since $ \dom \! \big( \Lambda(-1) \big) \subset \dom(L^{1/4}) $, 
	\begin{align*}
		\int_{\sigma}^{\oplus} 
		(1 + \lambda)^{1/4}
		\varphi(\lambda) \dd\mu(\lambda) 
		\in \mathcal{G}.
	\end{align*}
	Consequently, for all $ \eta \in \dom(L^{1/4}) $, 
	Lemma~\ref{Die_Dirichlet-zu-Neumann-Form} implies:
	\begin{align*}
		0 
		&= \mathfrak{l}_{-1} \Big( \varphi, \int_{\sigma}^{\oplus} 
		(1 + \lambda)^{-1/4}
		\eta(\lambda) \dd\mu(\lambda)  \Big)
		- \Big\langle \zeta, \int_{\sigma}^{\oplus} 
		(1 + \lambda)^{-1/4}
		\eta(\lambda) \dd\mu(\lambda) 
		\Big\rangle_{\mathcal{G}}	\\
		&= \Big\langle 
		\int_{\sigma}^{\oplus} 
		(1 + \lambda)^{1/4}
		\varphi(\lambda) \dd\mu(\lambda)
		- \int_{\sigma}^{\oplus} 
		(1 + \lambda)^{-1/4}
		\zeta(\lambda) \dd\mu(\lambda), 
		\eta
		\Big\rangle_{\mathcal{G}}.
	\end{align*}
	As $ \dom(L^{1/4}) $ is dense in $ \mathcal{G} $ we obtain that,
	for $ \dd \mu $-almost all $ \lambda $ in $ \sigma $,
	\begin{align*}
		(1 + \lambda)^{1/4} \varphi(\lambda)
		= (1 + \lambda)^{-1/4} \zeta(\lambda). 
	\end{align*}
	Therefore, $ \int_{\sigma}^{\oplus} 
	(1+\lambda)^{1/2} \, \varphi(\lambda) \dd\mu(\lambda)
	= \zeta \in \mathcal{G} $ and thus $ \varphi \in \dom(L^{1/2}) $, 
	as claimed.
\end{proof}
In particular, for all $ z \in \mathbb{C} \setminus [\min \sigma ,\infty) $,
the \emph{Neumann-to-Dirichlet operator}
\begin{align}	\label{Der_Neumann-zu-Dirichlet-Operator}
	\Lambda(z)^{-1}\colon \mathcal{G} \rightarrow \mathcal{G},	\quad
	\Lambda(z)^{-1} \varphi
	= \int_{\sigma}^{\oplus} 
	\frac{\im}{\sqrt{z - \lambda}}
	\varphi(\lambda) \dd\mu(\lambda),
\end{align}
is bounded.

%-----------------------------------------------------------------------
\subsection{A Krein-type resolvent formula}
\label{subsec:krein.res}
%-----------------------------------------------------------------------

We have now computed the extended solution operator as well as 
its adjoint and the Neumann-to-Dirichlet operator. 
Putting these results together we obtain, 
since the boundary pair $ (\Gamma, \mathcal{G}) $ is 
elliptically regular, 
the following Krein-type resolvent formula for
$ (H - z)^{-1} - (\HDir - z)^{-1} $. 
\begin{proposition}	
\label{Darstellungsformel_fuer_die_Differenz_der_Resolventen}
	Let $ z \in \mathbb{C} \setminus [\min \sigma, \infty) $. 
	Then $ (H - z)^{-1} - (\HDir - z)^{-1} 
	: \mathcal{H} \rightarrow \mathcal{H} $ satisfies
	\begin{align}	\label{Formel_Resolventendifferenz}
		(H - z)^{-1} - (\HDir - z)^{-1}
		=
		\bar{S}(z) \Lambda(z)^{-1} \big( \bar{S}(\bar{z}) \big)^{\ast}.
	\end{align}
	This operator acts on elementary tensors as follows:
	\begin{align*}
		\Big( 
		\bar{S}(z) \Lambda(z)^{-1} \big( \bar{S}(\bar{z}) \big)^{\ast} 
		(\psi \otimes \varphi)
		\Big)(t)	
		= \int_{\sigma}^{\oplus} 
		\frac{\im}{\sqrt{z-\lambda}}
		\varphi(\lambda)
		\int_{\mathbb{R}_{+}}
		\psi(\tau)
		\exp \! \big( \im \sqrt{z-\lambda} \, (t+\tau) \big)
                 \dd\tau \,\dd\mu(\lambda)
	\end{align*}
	for all $ \psi \in \mathsf C_{\mathrm c}(\mathbb{R}_{+}) $ 
	and all $ \varphi \in \mathcal{G} $. 
	Consequently, the difference of the resolvents from~\eqref{Formel_Resolventendifferenz} can be 
	evaluated explicitly on the dense subspace
	$ \mathsf C_{\mathrm c}(\mathbb{R}_{+}) \odot \mathcal{G} $ 
	of $ \mathcal{H} $.
\end{proposition}
\begin{proof}
  By Lemma~\ref{Randpaar_ist_elliptisch_regulaer} we know that the
  boundary pair $ (\Gamma, \mathcal{G}) $ is elliptically regular.
  Therefore,~\cite[Theorem 1.2]{Post_I}
  implies~\eqref{Formel_Resolventendifferenz}.  The explicit
  representation of~\eqref{Formel_Resolventendifferenz} on $ \mathsf
  C_{\mathrm c}(\mathbb{R}_{+}) \odot \mathcal{G} $ follows directly
  from Lemma~\ref{Der_fortgesetzte_Loesungsoperator__II},
  \eqref{Der_Neumann-zu-Dirichlet-Operator}, and
  Lemma~\ref{Der_Adjungierte_des_fortgesetzten_Loesungsoperators}.
\end{proof}

%-----------------------------------------------------------------------
\subsection{Explicit formulas for the boundary pair of the generalised
  half-space problem}
\label{sec:expl.formula.bd2}
%-----------------------------------------------------------------------

Let us summarise the explicit formulas we have found for the boundary
pair of the generalised half-space problem, written in a more handy
version without refering to the direct integral representation of $L$:
\begin{proposition}
  \label{Resultate_zum_Randpaar-Modell}	
  Let $ z \in \mathbb{C} \setminus [\min \sigma, \infty) $. One has:
  \begin{enumerate}
  \item The solution operator $ S(z)\colon \dom(L^{1/4}) \rightarrow
    \mathcal{H}^{1} $ is given by\footnote{In the case when $ L $ is
      bounded, cf.\,\cite[equation (4.3)]{Malamud_Neidhardt}.}
    \begin{equation*}
      \big( S(z) \varphi \big)(t) 
      = \exp \! \big( \mathrm{i} \sqrt{z - L} \, t \big) 
			\varphi.
    \end{equation*}
    In particular, $ \| S(-1) \varphi \|_{\mathcal{H}}^{2}
    \leq \frac{1}{2} \| \varphi \|_{\mathcal{G}}^{2} $ for every 
    $ \varphi \in \dom(L^{1/4}) $ so $ (\Gamma, \mathcal{G}) $ is 
    an elliptically regular boundary pair.
  \item The Dirichlet-to-Neumann operator $ \Lambda(z)\colon
    \dom(L^{1/2}) \rightarrow \mathcal{G} $ is given by\footnote{In
      the case when $ L $ is bounded, cf.\,\cite[Lemma
      4.2]{Malamud_Neidhardt}.}
    \begin{equation*}
      \Lambda(z) \varphi  
      = \mathrm{i} \sqrt{z - L} \, \varphi.
    \end{equation*}
	\item \label{Res.diff._allg.z_allg.L}
		The difference of the resolvents of $ H $ and $ \HDir $ 
		acts on elementary tensors as follows:
    \begin{multline*}
      \Big( 
			\big[ (H-z)^{-1} - (\HDir-z)^{-1} \big]
			(\psi \otimes \varphi) 
      \Big)(t)	\\
      = \im \int_{\mathbb{R}_{+}} 
      \psi(\tau) \exp \! \big( \im \sqrt{z-L} (t+\tau) \big)
      (\sqrt{z-L})^{-1} \varphi \dd \tau
    \end{multline*}
    for all $ \psi \in \mathsf C_{\mathrm c}(\mathbb{R}_{+}) $ 
    and all $ \varphi \in \mathcal{G} $.
  \end{enumerate}
\end{proposition}
\begin{proof}
  The results from Lemma~\ref{Loesungsoperator_I},
  Proposition~\ref{Berechnung von ran Gamma},
  Lemma~\ref{Randpaar_ist_elliptisch_regulaer},
  Lemma~\ref{Dirichlet-zu-Neumann-Operators__I},
  Lemma~\ref{Berechnung_des_Dirichlet-zu-Neumann-Operators}, and
  Proposition~\ref{Darstellungsformel_fuer_die_Differenz_der_Resolventen}
  carry over to the situation when $L$ is not necessarily a
  multiplication operator, using
  Theorem~\ref{Theorem_von_Neumann_direct_integral} and the functional
  calculus.
\end{proof}
\begin{proof}[Proof of Theorem~\ref{thm:main1}~\eqref{thm:main1__Teil(a)}]
  Set $z=-1$ in
  Proposition~\ref{Resultate_zum_Randpaar-Modell}~\eqref{Res.diff._allg.z_allg.L}.
\end{proof}

%-----------------------------------------------------------------------
%
% 444
\section{A formula with separated variables for the difference of the
  resolvents}
\label{sec:main.thm.1}
%
%-----------------------------------------------------------------------

In this section we will establish 
Theorem~\ref{Resolventendifferenz_unitaer_aequivalent}.  The outline
of the proof is as follows:
\paragraph{Step 1.}
We change the order of evaluation with respect to the variables $ t
\in \mathbb{R}_{+} $ and $ \lambda \in \sigma $ in the representation
formula from
Proposition~\ref{Darstellungsformel_fuer_die_Differenz_der_Resolventen}.
Then, for $ \dd \mu $-almost all $ \lambda $ in $ \sigma $, we will
obtain a vector-valued Hankel-type integral operator.
\paragraph{Step 2.}
The application of a scaling transformation will lead to a unitarily
equivalent representation of~\eqref{Formel_Resolventendifferenz} with
seperated variables, as claimed.

Step 1 will be performed in 
Subsection~\ref{subsec:main_thm_Step_1} and
Step 2 will be performed in 
Subsection~\ref{subsec:main_thm_Step_2}.
Finally, in 
Subsection~\ref{subsec:deduce_corollary_from_main_thm}, 
we will deduce 
Corollary~\ref{corollary_from_main_thm__spectral.properties} from 
Theorem~\ref{Resolventendifferenz_unitaer_aequivalent}.

%-----------------------------------------------------------------------
\subsection{Proof of
  Theorem~\ref{Resolventendifferenz_unitaer_aequivalent}. Step 1}
\label{subsec:main_thm_Step_1}
%-----------------------------------------------------------------------

First, we observe that
\begin{align*}
  W\colon 
  \mathsf C_{\mathrm c}(\mathbb{R}_{+}) \odot \mathcal{G} 
  \subset \mathcal{H}
  \rightarrow 
  \int_{\sigma}^{\oplus} 
  \mathsf L_2(\R_+)\otimes \mathcal{G}(\lambda)
  \dd\mu(\lambda), 
  \quad
  W(\psi \otimes \varphi) 
  = \int_{\sigma}^{\oplus} \psi \otimes \varphi(\lambda)
  \dd\mu(\lambda),
\end{align*}
defines an isometric operator with dense range.  We denote the unique
bounded extension of $W$ to $ \mathcal{H} $ by the same symbol $W$.
Obviously, $ W $ is a unitary operator from $ \mathcal{H} $ onto $
\int_{\sigma}^{\oplus} \mathsf L_2(\R_+) \otimes \mathcal{G}(\lambda)
\dd\mu(\lambda) $.  The similarity transformation with respect to the
natural unitary operator $ W $ leads to the expected result:
%-----------------------------------------------------------------------
\begin{lemma}
  \label{lem:res.diff}
  If $ z \in \mathbb{C} \setminus [\min \sigma ,\infty) $ then, for
  all $ \psi \in \mathsf C_{\mathrm c}(\mathbb{R}_{+}) $ and all $
  \varphi \in \mathcal{G} $, one has
  \begin{align*}
    \Big( W 
    \bar{S}(z) \Lambda(z)^{-1} \big( \bar{S}(\bar{z}) \big)^{\ast}
    (\psi \otimes \varphi)
    \Big)(\lambda)	
    = 
    \frac{\im}{\sqrt{z-\lambda}}
    \int_{\mathbb{R}_{+}} 
    \psi(\tau)
    \exp \! \big( \im \sqrt{z-\lambda} \, (\bullet+\tau) \big)
    \dd\tau
    \otimes \varphi(\lambda)
  \end{align*}
  for $ \dd \mu $-almost all $ \lambda $ in $ \sigma $.
\end{lemma}
%-----------------------------------------------------------------------
\begin{proof}
  This is a consequence of
  Proposition~\ref{Darstellungsformel_fuer_die_Differenz_der_Resolventen}
  and Fubini's theorem.
\end{proof}
%-----------------------------------------------------------------------
In particular, Lemma~\ref{lem:res.diff} shows that
\begin{align*}
	W^{-1} 
	\big[ (H - z)^{-1} - (\HDir - z)^{-1} \big]
	W
	= \int_{\sigma}^{\oplus}
	T_{\lambda} \dd \mu(\lambda),
\end{align*}
where
\begin{align}	\label{Vektorwertiger_Hankel-Operator_aus_Res.diff.}
  T_{\lambda}( \psi \otimes \phi_\lambda ) 
  = \frac{\im}{\sqrt{z - \lambda}} 
  \int_{\mathbb{R}_{+}} 
  \psi(\tau) 
  \exp \! \big( \im \sqrt{z-\lambda} \, (\bullet+\tau) \big)
  \dd\tau 
  \otimes \phi_\lambda
\end{align}
for all $ \psi \in \mathsf C_{\mathrm c}(\mathbb{R}_{+}) $, all
$\phi_\lambda \in \mathcal{G}(\lambda)$ and $\dd \mu$-almost all
$\lambda \in \sigma$.  We write $ T = \int_{\sigma}^{\oplus}
T_{\lambda} \dd \mu(\lambda) $.
\begin{remark}
  The operator $ T_{\lambda} $ defined in~\eqref{Vektorwertiger_Hankel-Operator_aus_Res.diff.}  is a
  vector-valued Hankel-type integral operator, as the first factor is
  an integral operator on $\mathsf L_2(\R_+)$ with kernel depending
  only on the sum $t+\tau$ of the variables $t,\tau \in \R_+$.
\end{remark}

%-----------------------------------------------------------------------
\subsection{Proof of
  Theorem~\ref{Resolventendifferenz_unitaer_aequivalent}. Step 2}
\label{subsec:main_thm_Step_2}
%-----------------------------------------------------------------------

For the rest of this subsection we assume that 
\begin{equation*}
  z \in (-\infty, \min \sigma).
\end{equation*}
It is then clear that $ \lambda - z > 0 $ and hence $ \im \sqrt{z -
  \lambda} = -(\lambda - z)^{1/2}$ for all $ \lambda \in \sigma $.
Therefore,
\begin{align*}
  %\label{eq:unitary_map}
  U_{\lambda}\colon \mathsf L_2\big( \mathbb{R}_{+}, \mathcal{G}(\lambda) \big)
	\rightarrow \mathsf L_2\big( \mathbb{R}_{+}, \mathcal{G}(\lambda) \big),
	\quad
	(U_{\lambda}f)(t) 
	= (\lambda - z)^{1/4} f \big( (\lambda - z)^{1/2} \, t \big),
\end{align*}
is a unitary operator for every fixed $ \lambda $ outside a set of 
$ \dd \mu $-measure $0$, and 
the operator-valued function 
$ U = \int_{\sigma}^{\oplus} U_{\lambda} \dd \mu(\lambda) $
defines a unitary operator on
$ \int_{\sigma}^{\oplus} 
\mathsf L_2\big( \mathbb{R}_{+}, \mathcal{G}(\lambda) \big)
\dd\mu(\lambda) $.   Note that
$U$ depends on $z$, but we will suppress the dependency of $z$ in the
notation (as we already did for $T$ in the previous subsection).

Let us now perform the scaling transformation of $ T $ with respect to
$ U $.  As both operators are fibred with respect to the direct
integral over $\lambda$, we have $U^{-1}TU=\int_\sigma^\oplus
U_\lambda^{-1}T_\lambda U_\lambda \dd \mu(\lambda)$.  Moreover, for
$\psi \in \mathsf C_{\mathrm c}(\mathbb{R}_{+}) $ and $
\varphi_\lambda \in \mathcal{G}(\lambda) $ we calculate
\begin{align*}	%\label{Variablen_nun_getrennt__Res.diff.}
  (U_\lambda^{-1} T_\lambda U_\lambda)(\psi \otimes \varphi_\lambda)
  = \int_{\R_+}
  \exp \! \big( {-} (\bullet+\tau) \big) \psi(\tau) 
	\dd\tau 
	\otimes  \frac{\varphi_\lambda}{\lambda-z}
\end{align*}
for $ \dd \mu $-almost all $ \lambda \in \sigma$.

Let $\map {\Psi_0} {\R_+} \R$ be the function defined by $\Psi_{0}(t)
= \exp(-t)$. It is well known that the difference of the resolvents
(at $ -1 $) of the Neumann and Dirichlet Laplacians on the semi-axis
is given by
\begin{align}	\label{Res.diff._Halbgerade}
	\bigg[ \Big( {-} \frac{\dd^{2}}{\dd t^{2}} \Big)^{\Neu} + 1 \bigg]^{-1} 
	- \bigg[ \Big( {-}\frac{\dd^{2}}{\dd t^{2}} \Big)^{\Dir} + 1 \bigg]^{-1} 
	= \langle \bullet, \Psi_{0} \rangle_{\mathsf L_2(\mathbb{R}_{+})} \Psi_{0}.
\end{align}
Since $ L $ is the multiplication operator by the independent variable 
on $ \mathcal{G} $ one has
\begin{align*}
	(L - z)^{-1} \varphi 
	= \int_{\sigma}^{\oplus}
	\frac{\varphi(\lambda)}{\lambda - z}
	\dd \mu(\lambda).
\end{align*}
We have thus shown 
Theorem~\ref{Resolventendifferenz_unitaer_aequivalent}.	\qed

%-----------------------------------------------------------------------
\subsection{The spectral properties of the difference of the 
resolvents}
\label{subsec:deduce_corollary_from_main_thm}
%-----------------------------------------------------------------------

Theorem~\ref{Resolventendifferenz_unitaer_aequivalent} allows us to
determine the spectral properties of the difference of the resolvents
as stated in
Corollary~\ref{corollary_from_main_thm__spectral.properties}.
\begin{proof}[Proof of
  Corollary~\ref{corollary_from_main_thm__spectral.properties}]
	Denote by $ B_{\Psi_{0}} $ the self-adjoint rank one operator 
	on $ \mathsf L_2(\mathbb{R}_{+}) $ from equation~\eqref{Res.diff._Halbgerade},
	where $ \Psi_{0}(t) = \exp(-t) $. 
	By Theorem~\ref{Resolventendifferenz_unitaer_aequivalent} 
	we know that 
	$ (H - z)^{-1} - (\HDir - z)^{-1} $ on $ \mathcal{H} $ 
	is unitarily equivalent to 
	\begin{align} \label{reformulation_of_res.diff.}
		B_{\Psi_{0}}
		\otimes (L-z)^{-1} 
		\quad \text{on } \mathsf L_2(\mathbb{R}_{+}) \otimes \mathcal{G}.
	\end{align}
	Denote by $ \{ \Psi_{0} \}^{\bot} $ the orthogonal complement of 
	$ \C \Psi_{0} $ in $ \mathsf L_2(\mathbb{R}_{+}) $.
	Then the operator from~\eqref{reformulation_of_res.diff.}
	is unitarily equivalent to the 
	block diagonal operator
	\begin{align*}
		0 \oplus \Big[ \frac{1}{2} (L-z)^{-1} \Big] 
		\quad \text{on }
		\big[ \{ \Psi_{0} \}^{\bot} \otimes \mathcal{G} \big]
		\oplus \mathcal{G},
	\end{align*}
	because the range of $ B_{\Psi_{0}} $ is spanned by $ \Psi_{0} $
	and $ \langle \Psi_{0}, \Psi_{0} 
	\rangle_{\mathsf L_2(\mathbb{R}_{+})} = \frac{1}{2} $.
	Now, standard arguments from spectral theory (see, e.\,g., 
	\cite[Chapter 7]{Birman_Solomyak}) complete the proof.
\end{proof}

%-----------------------------------------------------------------------
%
% 555
\section{The difference of the spectral projections}
\label{sec:main.thm.2}
%
%-----------------------------------------------------------------------

In this section we will establish
Theorem~\ref{thm:main1}~\eqref{thm:main1__Teil(b)} and
Theorem~\ref{Ergebnis_I_Differenz_der_Spektralprojektoren}.  In
Subsection~\ref{subsec:main_thm_1.1} we use Proposition
\ref{Allgemeine_Fakten_zum_Spektrum_von_H} to compute the difference $
\1_{(-\infty,\alpha)}(H) - \1_{(-\infty, \alpha)}(\HDir) $ of the
spectral projections for every $ \alpha > 0 $. Then, we will establish
Theorem~\ref{thm:main1}~\eqref{thm:main1__Teil(b)}.  In
Subsection~\ref{subsec:main_thm_2__Step_1} we will change the order of
evaluation with respect to the variables $ t \in \mathbb{R}_{+} $ and
$ \lambda \in \sigma $ in the formula for $ \1_{(-\infty,\alpha)}(H) -
\1_{(-\infty, \alpha)}(\HDir) $.  We will obtain, for $ \dd \mu
$-almost all $ \lambda $ in $ \sigma $, a vector-valued Hankel-type
integral operator.  In Subsection~\ref{subsec:main_thm_2__Step_2} we
will see that these vector-valued Hankel-type integral operators are
closely related to the Hankel integral operator from Krein's
example~\cite{Krein} discussed above. After this observation we will
be able to complete the proof of
Theorem~\ref{Ergebnis_I_Differenz_der_Spektralprojektoren}, using the
above mentioned result from Kostrykin and
Makarov~\cite{Kostrykin_Makarov}.

%-----------------------------------------------------------------------
\subsection{Proof of
  Theorem~\ref{thm:main1}~\eqref{thm:main1__Teil(b)}}
\label{subsec:main_thm_1.1}
%-----------------------------------------------------------------------
 
Since $ H \geq 0 $ and $ \HDir \geq 0 $ both have a purely 
absolutely continuous spectrum we may, 
without loss of generality, assume that $ \alpha > 0 $. 
By Proposition~\ref{Allgemeine_Fakten_zum_Spektrum_von_H}~\eqref{Weimann-Formel}, formula~\eqref{eq:spec.proj.class1},
and Fubini's theorem
we obtain that
\begin{align*}
	&\big\langle
	\big( \1_{(-\infty, \alpha)}(H) - \1_{(-\infty, \alpha)}(\HDir) \big)
	(\psi \otimes \varphi),
	\xi \otimes \eta 
	\big\rangle_{\mathcal{H}}	\\
	&= \frac{2}{\pi} 
	\int_{\mathbb{R}_{+}} \dd t
	\int_{\sigma} \dd\mu(\lambda)
	\int_{\mathbb{R}_{+}} \dd\tau ~
	\langle 
	\psi(\tau) \1_{[0, \alpha)}(\lambda) \varphi(\lambda), 
	\xi(t) \eta(\lambda) \rangle_{\mathcal{G}(\lambda)}
	\frac{\sin \! \big( (\alpha - \lambda)^{1/2} \, (t+\tau) \big)}
	{t + \tau}
\end{align*}
for all $ \psi, \xi \in \mathsf C_{\mathrm c}(\mathbb{R}_{+}) $ and 
all $ \varphi \in \mathcal{G} $, $ \eta \in \dom(L) $.
\begin{remark}
  Alternatively, this can also be computed using
  Proposition~\ref{Darstellungsformel_fuer_die_Differenz_der_Resolventen}
  and Stone's formula for spectral projections.
\end{remark}

Further one proves for all $ t $ in $ \mathbb{R}_{+} $ that
\begin{align}	\label{Formel_fuer_Proj.diff._in_H}
	h(t)
	= \frac{2}{\pi} 
	\int_{\sigma}^{\oplus}
	\int_{\mathbb{R}_{+}} 
	\psi(\tau)
	\1_{[0, \alpha)}(\lambda)
	\frac{\sin \! \big( (\alpha - \lambda)^{1/2} (t+\tau) \big)}
	{t + \tau}
        \dd\tau
	\varphi(\lambda)
        \dd\mu(\lambda) 
        \in \mathcal{G}.
\end{align}
By the dominated convergence theorem we obtain that 
$ \mathbb{R}_{+} \ni t \mapsto h(t) \in \mathcal{G} $ is 
continuous. Consequently, $ h $ is measurable and we compute
\begin{align*}
	\| h \|_{\mathcal{H}} 
	\leq 
	\| \varphi \|_{\mathcal{G}} \,
	\frac{1}{\tau_{0}^{1/2}} \,
	\max_{\tau \in \mathbb{R}_{+}} |\psi(\tau)|  \,
	\int_{\supp(\psi)} \dd\tau
	< \infty,
\end{align*}
where $ \tau_{0} = \min \! \big( \supp(\psi) \big) > 0 $.
We have shown that
\begin{align*}
	\big\langle
	\big( \1_{(-\infty, \alpha)}(H) - \1_{(-\infty, \alpha)}(\HDir) \big)
	(\psi \otimes \varphi),
	\xi \otimes \eta 
	\big\rangle_{\mathcal{H}}
	= \langle h, \xi \otimes \eta \rangle_{\mathcal{H}}
\end{align*}
for all $ \xi \in \mathsf C_{\mathrm c}(\mathbb{R}_{+}) $ and 
all $ \eta \in \dom(L) $. 
Since $ \mathsf C_{\mathrm c}(\mathbb{R}_{+}) \odot \dom(L) $ is 
dense in $ \mathcal{H} $ we have established the following result:
\begin{lemma}	\label{Spektr.proj.diff.__L_Mult.op.}
	If $ \alpha > 0 $ then
	%\begin{align*}
	$	\big( \1_{(-\infty, \alpha)}(H) - \1_{(-\infty, \alpha)}(\HDir) \big)
		(\psi \otimes \varphi)
		= h $ 
	%\end{align*}
	for all $ \psi \in \mathsf C_{\mathrm c}(\mathbb{R}_{+}) $ and 
	all $ \varphi \in \mathcal{G} $,
	where $ h \in \mathcal{H} $ is defined as in~\eqref{Formel_fuer_Proj.diff._in_H} above.
\end{lemma}
We can now prove Theorem~\ref{thm:main1}~\eqref{thm:main1__Teil(b)}.
\begin{proof}[Proof of Theorem~\ref{thm:main1}~\eqref{thm:main1__Teil(b)}]
	The result from 
	Lemma~\ref{Spektr.proj.diff.__L_Mult.op.}
	carries over to the situation when $ L $ is not 
	necessarily a multiplication operator, using 
	Theorem~\ref{Theorem_von_Neumann_direct_integral} and 
	the functional calculus:
	\begin{align*}
		&\Big( \big[ \1_{(-\infty,\alpha(\vartheta))}(H) 
		- \1_{(-\infty,\alpha(\vartheta))}(\HDir) \big]
		(\psi \otimes \varphi) \Big)(t)	\\
		&= \frac{2}{\pi} \int_{\mathbb{R}_{+}}
		\psi(\tau)
		\1_{[0,\alpha(\vartheta))}(L) 
		\frac{\sin \! \big( (\alpha(\vartheta) - L)^{1/2} (t+\tau) \big)}{t+\tau}
		\varphi
		\dd \tau
	\end{align*}
	for all $ \psi \in \mathsf C_{\mathrm c}(\mathbb{R}_{+}) $ and 
	all $ \varphi \in \mathcal{G} $, where $ 0<\vartheta<1 $ 
	and $ \alpha(\vartheta) = \frac{1}{\vartheta}-1 $.
	Last, observe that
	\begin{align}	\label{Trafo_Spektralmasse}
		\1_{(-\infty,\alpha(\vartheta))}(H) 
		- \1_{(-\infty,\alpha(\vartheta))}(\HDir)
		= \1_{(-\infty,\vartheta)}(A_0)
		- \1_{(-\infty,\vartheta)}(A_1).
	\end{align}
	Now the proof of Theorem~\ref{thm:main1}~\eqref{thm:main1__Teil(b)}
	is complete.
\end{proof}

%-----------------------------------------------------------------------
\subsection{Proof of
  Theorem~\ref{Ergebnis_I_Differenz_der_Spektralprojektoren}. Step 1}
\label{subsec:main_thm_2__Step_1}
%-----------------------------------------------------------------------

Analogously to Lemma~\ref{lem:res.diff} one shows:
\begin{lemma}	\label{lem:spec.proj.diff}
	Let $ \alpha > 0 $ and let
	$ \psi \in \mathsf C_{\mathrm c}(\mathbb{R}_{+}) $, $ \varphi \in \mathcal{G} $.
	Then one has
	\begin{align*}
		&\Big( W 
		\big( \1_{(-\infty, \alpha)}(H) - \1_{(-\infty, \alpha)}(\HDir) \big)
		(\psi \otimes \varphi)
		\Big)(\lambda)	\\
		&= 
		\frac{2}{\pi} 
		\int_{\mathbb{R}_{+}} 
		\1_{[0, \alpha)}(\lambda)
		\psi(\tau)
		\frac{\sin \! \big( (\alpha - \lambda)^{1/2} (\bullet+\tau) \big)}
		{\bullet + \tau}
		\dd\tau
                \otimes \varphi(\lambda)
	\end{align*}
	for $ \dd \mu $-almost all $ \lambda $ in $ \sigma $, where $
        W \colon \mathcal{H} \rightarrow \int_{\sigma}^{\oplus}
        \mathsf L_2( \mathbb{R}_{+})\otimes \mathcal{G}(\lambda)
        \dd\mu(\lambda) $ is the unitary operator defined in
        Subsection~\ref{subsec:main_thm_Step_1} above.
\end{lemma}

%-----------------------------------------------------------------------
\subsection{Proof of
  Theorem~\ref{Ergebnis_I_Differenz_der_Spektralprojektoren}. Step 2}
\label{subsec:main_thm_2__Step_2}
%-----------------------------------------------------------------------

Lemma~\ref{lem:spec.proj.diff} shows that if $ \mu \big( \sigma \cap
[0,\alpha) \big) = 0 $ then 
$ \1_{(-\infty, \alpha)}(H) - \1_{(-\infty, \alpha)}(\HDir) = 0 $.  
Let us now consider the more interesting
case when $ \mu \big( \sigma \cap [0,\alpha) \big) > 0 $.
Lemma~\ref{lem:spec.proj.diff} implies in this case that
$	\1_{(-\infty, \alpha)}(H) - \1_{(-\infty, \alpha)}(\HDir) $
is unitarily equivalent to the block diagonal operator
\begin{align*}
	\Big[ \int_{\sigma \cap [0, \alpha)}^{\oplus}
	\widetilde{T}_{\lambda} \dd\mu(\lambda) \Big]
	\oplus 0
	\quad \text{on }
	\Big[ 
	\int_{\sigma \cap [0,\alpha)}^{\oplus} 
	\mathsf L_2(\mathbb{R}_{+}, \mathcal{G}(\lambda))
	\dd\mu(\lambda)
	\Big] 
	\oplus
	\Big[ 
	\int_{\sigma \cap [\alpha,\infty)}^{\oplus} 
	\mathsf L_2(\mathbb{R}_{+}, \mathcal{G}(\lambda))
	\dd\mu(\lambda)
	\Big],
\end{align*}
where for every fixed 
$ \lambda \in \sigma \cap [0, \alpha) $ outside a set of 
$ \dd \mu $-measure $0$
\begin{align*}
	\widetilde{T}_{\lambda}( \psi \otimes \phi_\lambda ) 
	= \frac{2}{\pi}  
	\int_{\mathbb{R}_{+}} 
	\psi(\tau) 
	\frac{\sin \! \big( (\alpha - \lambda)^{1/2} \, (\bullet+\tau) \big)}
	{\bullet + \tau}
	\dd\tau 
	\otimes \phi_\lambda
\end{align*}
for all $ \psi \in \mathsf C_{\mathrm c}(\mathbb{R}_{+}) $ and 
all vectors $ \phi_\lambda \in \mathcal{G}(\lambda) $. 
We will write $ \widetilde{T} 
= \int_{\sigma \cap [0, \alpha)}^{\oplus}
\widetilde{T}_{\lambda} \dd\mu(\lambda) $.

Next we define the unitary operator
\begin{equation*}
  \widetilde{U} 
  = \int_{\sigma \cap [0, \alpha)}^{\oplus} \widetilde{U}_{\lambda} 
      \dd\mu(\lambda)
  \quadtext{on}
  \int_{\sigma \cap [0,\alpha)}^{\oplus} 
     \mathsf L_2(\mathbb{R}_{+}, \mathcal{G}(\lambda)) \dd\mu(\lambda),
\end{equation*}
where $\widetilde{U}_{\lambda}$ is the unitary scaling operator on $
\mathsf L_2\big( \mathbb{R}_{+}, \mathcal{G}(\lambda) \big) $ given by
\begin{equation*}
  (\widetilde{U}_{\lambda}f)(t) 
  = (\alpha - \lambda)^{1/4} f \big( (\alpha - \lambda)^{1/2} \, t \big)
\end{equation*}
for $\dd \mu$-almost all $\lambda \in \sigma \cap [0,\alpha)$.  Note
that $\wt U$ depends also on $\alpha$, but as before for $U$, we
suppress this dependency.
Again, both operators $\wt U$ and $\wt T$ are fibred with respect to the direct
integral over $\lambda$, hence 
$\wt U^{-1}\wt T \wt U=\int_{\sigma \cap [0, \alpha)}^\oplus
\wt U_\lambda^{-1} \wt T_\lambda \wt U_\lambda \dd \mu(\lambda)$.  Moreover, for
$\psi \in \mathsf C_{\mathrm c}(\mathbb{R}_{+}) $ and $
\varphi_\lambda \in \mathcal{G}(\lambda) $ we compute
\begin{align*}
  (\wt U^{-1}_\lambda \wt T_\lambda \wt U_\lambda)(\psi \otimes \varphi_\lambda) 
  &= \frac{2}{\pi} 
  \int_{\mathbb{R}_{+}} 
  \frac{\sin(\bullet+\tau)}{\bullet+\tau} \psi(\tau)
  \dd\tau
  \otimes \varphi_\lambda\\
  &= K \psi \otimes \varphi_\lambda
\end{align*}
for $ \dd \mu $-almost all $ \lambda \in \sigma \cap [0, \alpha)$,
where $K$ is given by
\begin{align*}	%\label{Hankel-Operator_aus_Kreinschem_Bsp.}
    (K \psi)(t) = 
    \frac{2}{\pi} 
    \int_{\mathbb{R}_{+}} 
    \frac{\sin(t+\tau)}{t+\tau} \psi(\tau)
    \dd\tau,
    \quad 
    \psi \in \mathsf C_{\mathrm c}(\mathbb{R}_{+}).
\end{align*}
In~\cite{Kostrykin_Makarov}, Kostrykin and Makarov have shown that $ K
$ has a simple and purely absolutely continuous spectrum filling in
the interval $ [-1,1] $.  Consequently, the operator
\begin{align*}
	\widetilde{U}^{-1} \, \widetilde{T} \, \widetilde{U} 
	\quad	\text{on }
	\int_{\sigma \cap [0,\alpha)}^{\oplus} 
	\mathsf L_2(\mathbb{R}_{+}) \otimes \mathcal{G}(\lambda)
	\dd\mu(\lambda)
\end{align*}
is unitarily equivalent to the multiplication operator by the
independent variable on
\begin{equation*}
  \mathsf L_2 \Big( [-1,1], \int_{\sigma \cap
  [0,\alpha)}^{\oplus} \mathcal{G}(\lambda) \dd\mu(\lambda) \Big).
\end{equation*}
Now, an application of the transformation rule for spectral measures
(see~\eqref{Trafo_Spektralmasse} above) completes to proof 
of Theorem~\ref{Ergebnis_I_Differenz_der_Spektralprojektoren}. \qed

\begin{remark}
  \label{rem:diff.is.hankel}
  Note that $K$ defined above is the Hankel integral operator on $
  \mathsf L_2(\mathbb{R}_{+})$ from Krein's example
	in the case when $ \vartheta = 1/2 $,
	see (\ref{eq:spec.proj.class1}) above.
\end{remark}

%----------------------------------------------------------------------
% yyyy
%
% Bibliography
%
%----------------------------------------------------------------------

\providecommand{\bysame}{\leavevmode\hbox to3em{\hrulefill}\thinspace}

%\tableofcontents

\end{document}